\documentclass[largeextended,15pt]{svjour3}       
\smartqed  

\usepackage[sans]{dsfont}
\usepackage{enumitem}
\usepackage{cite}
\usepackage{multirow}
\usepackage{latexsym, epsfig, amsmath, amssymb, amsfonts}
\usepackage{mathrsfs, graphics, fancyhdr, setspace, color}
\usepackage[bookmarksnumbered, plainpages, backref, colorlinks]{hyperref}
\usepackage{graphicx, latexsym}
\usepackage{mathptmx}      
\usepackage{float}
\usepackage{subcaption}
\usepackage{algpseudocode}  

\usepackage{rotating}    
\usepackage{booktabs}    
\usepackage{caption}     

\newtheorem{algorithm}[theorem]{Algorithm}
\newtheorem{assumption}[theorem]{Assumption}

\journalname{}

\begin{document}
	
\title{Non-Lipschitz Inertial Contraction-Type Method for Monotone Variational Inclusion problems}
	
\titlerunning{Non-Lipschitz Inertial Contraction-Type Method for Monotone Variational Inclusion problems}
	
\author{Feeroz Babu \and Syed Shakaib Irfan \and  Jen-Chih Yao \and Xiaopeng Zhao}
	
	
\authorrunning{F. Babu et al.} 
	
\institute{Feeroz Babu \at Mathematics Division, School of Advanced Sciences and Languages, VIT Bhopal University, Kothrikalan, Sehore, Madhya Pradesh – 466114, India
\email{firoz77b@gmail.com}
\and
Syed Shakaib Irfan \at Department of Mathematics, Aligarh Muslim University, Aligarh 202 002, India
\email{ssirfan.mm@amu.ac.in}
\and
Jen-Chih Yao \at Research Center for Interneural Computing, China Medical University Hospital, China Medical University, Taichung, 40447, Taiwan and Academy of Romanian Scientists, Bucharest, 50044, Romania
\email{yaojc@mail.cmu.edu.tw}
\and
Xiaopeng Zhao \at School of Mathematical Sciences, Tiangong University, Tianjin, 300387, China
	\email{zhaoxiaopeng.2007@163.com}}
	
\date{Received: \today / Accepted: date}

\maketitle
	

\begin{abstract}
	This study explores an inertial-based contraction-type approach for addressing monotone variational inclusion problems (in short, MVIP) within real Hilbert spaces. Most contraction-type techniques assume Lipschitz continuity and monotonicity or co-coercivity (inverse strongly monotone) of the single-valued operator. However, the key advantage of the proposed method is that it does not rely on the coercivity condition and the Lipschitz continuity for the single-valued operator. A weak convergence result has been achieved for the proposed algorithm with a convergence rate $\mathcal{O}\left(1/\sqrt{k}\right)$. In addition, the maximal and strong monotonicity of the set-valued operator is used to establish a strong convergence result with the linear convergence rate. To demonstrate the effectiveness of our proposed method, we conduct numerical experiments focused on signal recovery problems. 
\end{abstract}
	
\keywords{Hilbert space; monotone operator; Variational inclusion problems; Inertial method; Resolvent operator; Weak convergence; Linear convergence rate}

\subclass{47J25; 90C30; 65K15}
			
	
\section{Introduction}\label{S:1}

	
	\noindent  In recent years, variational inclusion problems have emerged as fundamental tools in mathematical programming, optimization, control theory, and various applications in the field of image processing and machine learning. These problems provide a versatile framework for unifying and solving many optimization-related challenges.

Given a real Hilbert space $\mathcal{H}$ with inner product $ \langle \cdot, \cdot \rangle $ and induced norm $ \|\cdot\| $, let $\mathcal{A}: \mathcal{H} \rightrightarrows \mathcal{H}$ be a set-valued operator and $ \mathcal{B}: \mathcal{H} \to \mathcal{H}$ be a single-valued operator. The domain of $\mathcal{A}$ is denoted by $D(\mathcal{A}) = \{ u\in \mathcal{H} : \mathcal{A}u \neq \emptyset\}$. The monotone variational inclusion problem (in short, MVIP) is expressed as follows.
\begin{equation}\label{MVIP}
	\mbox{Find } u^* \in (\mathcal{A}+\mathcal{\mathcal{B}})^{-1}(0).
\end{equation}
When $\mathcal{B}\equiv 0 $, problem \eqref{MVIP} reduces to the inclusion problem introduced by Rockafellar \cite{R76}. MVIPs are essential for optimization, variational inequalities, equilibrium models, and optimal control. Moreover, their structure facilitates the development of iterative algorithms for finding solutions, particularly using resolvent operators. Specifically, $ u^* \in \mathcal{H} $ is a solution of MVIP \eqref{MVIP}  if and only if it is the fixed point of the resolvent operator
$$ 
J_{\lambda\mathcal{A}} (I - \lambda\mathcal{B}) := (I+\lambda \mathcal{A})^{-1} (I - \lambda\mathcal{B}), \quad \lambda > 0.
$$
This is the motivation behind various iterative methods to solve the problem of MVIP.   In the case where the operator $\mathcal{A}: \mathcal{H} \rightrightarrows \mathcal{H}$ is \textit{maximal monotone} and $\mathcal{B} : \mathcal{H} \to \mathcal{H}$ is $\beta$-cocoercive, one of the classical and widely used methods for solving monotone variational inclusion problems (MVIP) is the forward-backward algorithm (in short, FBA). This method, originally proposed by Lions and Mercier \cite{LM79}, generates a sequence $\{u_k\}$ by iteration.
\begin{equation}\label{FB}
	u_{k+1} = J_{\lambda_k\mathcal{A}} \left( I - \lambda_k\mathcal{B}\right)(u_k), \quad \forall k \in \mathbb{N},
\end{equation}
Under appropriate assumptions, the sequence $\{u_k\}$ converges weakly to the MVIP.

Many scholars have developed significant results for MVIP in recent years, assuming that the operators are strongly monotone or inversely strongly monotone; see, e.g. \cite{MT11,S20,TWY12,IOM20,SAY15,GSDS21} and the references therein. Huang~\cite{H98} studied the MVIP \eqref{MVIP} in a setting where $\mathcal{A}$ is maximal monotone and $\mathcal{B}$ is strongly monotone and Lipschitz continuous and proved the existence of the solution and the convergence of the iterative method. Later, Zeng et al.~\cite{ZGY05} introduced a new iterative algorithm to solve a class of variational inclusions and established strong convergence results under appropriate parameter conditions.

However, in many practical problems, the operator $\mathcal{B}$ may not satisfy the assumptions of strong monotonicity, inverse strong monotonicity, or Lipschitz continuity. In particular, the classical forward-backward splitting algorithm \eqref{FB} typically requires the operator $\mathcal{B}$ to be strongly inverse monotone, which is often too restrictive for real-world applications. Therefore, relaxing these conditions has become crucial in solving more general MVIPs. In this direction, various authors~\cite{AOM23,M18,OMUCN24} have proposed projection and contraction methods, initially in Euclidean spaces and later extended to Hilbert spaces, to address these challenges and broaden the applicability of iterative methods.

An important development in this direction is the Tseng splitting algorithm \cite{T00}. This two-step iterative scheme is given by
\begin{equation}\label{Tseng}
	\begin{cases}
		v_k &= J_{\lambda_k\mathcal{A}}(I - \lambda_k \mathcal{B})u_k, \\
		u_{k+1} &= v_k - \lambda_k (\mathcal{B} v_k - \mathcal{B} u_k),
	\end{cases}	
\end{equation}
where step sizes $\{\lambda_k\}$ can be updated automatically using Armijo-type line search strategies. Under the assumption that $\mathcal{A}$ is maximal monotone and $\mathcal{B}$ is Lipschitz continuous and monotone, the sequence $\{u_k\}$ generated by \eqref{Tseng} converges weakly to a solution of MVIP in real Hilbert spaces. Furthermore, Zhang and Wang~\cite{ZW18} proposed a hybrid iterative method combining projection and contraction techniques with the Tseng splitting idea, described as
\begin{equation}\label{ZW}
	\begin{cases}
		v_k &= J_{\lambda_k\mathcal{A}}(I - \lambda_k \mathcal{B})u_k, \\
		\phi(u_k,\,v_k) &= (u_k - v_k) - \lambda_k (\mathcal{B} u_k - \mathcal{B} v_k), \\
		\alpha_k &= \frac{\langle u_k - v_k, \phi(u_k,\,v_k) \rangle}{\|\phi(u_k,\,v_k)\|^2}, \\
		u_{k+1} &= u_k - \gamma \alpha_k \phi(u_k,\,v_k),
	\end{cases}	
\end{equation}
where $\gamma \in (0,2)$. Under suitable assumptions, they established the weak convergence of the generated sequence $\{u_k\}$ when $\mathcal{A}$ is maximal monotone and $\mathcal{B}$ is Lipschitz continuous and monotone.

The concept of inertial extrapolation in iterative algorithms dates back to the pioneering work of Polyak~\cite{P64}, who introduced the heavy ball method based on a second-order dynamical system to accelerate the convergence of smooth convex minimization problems. Building on this idea, many researchers have developed various fast iterative algorithms using inertial extrapolation techniques; see, e.g. \cite{S20,OMUCN24, TC21,WLC24,TV19,AOM23,YIS22} and references therein.

More recently, Lorenz and Pock~\cite{LP15} introduced an inertial forward-backward algorithm to solve MVIP. Their method generates the sequence $\{u_k\}$ according to
\begin{equation}\label{IFBA}
	\begin{cases}
		v_k &= u_k + \vartheta_k (u_k - u_{k-1}), \\
		u_{k+1} &= J_{\lambda_k\mathcal{A}}(I - \lambda_k \mathcal{B})v_k,
	\end{cases}	
\end{equation}
where $\vartheta_k\in (0, \vartheta)$, $\vartheta >0$ and $\lambda_k>0$ are appropriately chosen parameters. This inertial scheme takes advantage of information from two consecutive iterations to accelerate convergence, like the previous ones. As a result, inertial-based methods have become a major research direction in designing fast algorithms for MVIP and related optimization problems.

Recently, Tan and Cho \cite{TC21} proposed the following inertial viscosity-type projection algorithm for MVIP in Hilbert spaces
\begin{equation}\label{BC21}
	\begin{cases}
		\vartheta_k &= 
		\begin{cases}
			\min\left\{ \frac{\epsilon_k}{ \|u_k - u_{k-1}\|}, \vartheta \right\}, & \text{if } u_k \neq u_{k-1}, \\
			\vartheta, & \text{otherwise}.
		\end{cases}\\
		w_k &= u_k + \vartheta_k (u_k - u_{k-1}),\\
		v_k &= J_{\lambda_k\mathcal{A}}(I - \lambda_k \mathcal{B})u_k, \\
		\phi(w_k,\,v_k) &= (w_k - v_k) - \lambda_k (\mathcal{B} w_k - \mathcal{B} v_k), \\
		\eta_k &= \frac{(1 - \mu)\|w_k - v_k\|^2}{\|\phi(w_k,\,v_k)\|^2}, \\
		z_k &= w_k - \gamma \eta_k \phi(u_k,\,v_k),\\
		u_{k+1} &= \alpha_k f(u_k) + (1 - \alpha_k) z_k.
	\end{cases}	
\end{equation}
where $\delta > 0$, $\vartheta > 0$, $l \in (0,1)$, $\mu \in (0,1)$, $\gamma \in (0,2)$, $\{\alpha_k\} \subset (0,1)$ and $\{\epsilon_k\}\in (0, \infty)$. Under assumptions, when $\mathcal{A}: \mathcal{H} \rightrightarrows \mathcal{H}$ is maximal monotone, $\mathcal{B}: \mathcal{H} \to \mathcal{H}$ is Lipschitz continuous and monotone and $f : \mathcal{H} \to \mathcal{H} $ is $\rho$-contraction with constant $\rho \in [0,1)$, $\lim\limits_{k \to \infty} \frac{\epsilon_k}{\alpha_k} = 0$, $\lim\limits_{k \to \infty} \alpha_k = 0$ and $ \sum\limits_{k=1}^{\infty} \alpha_k = \infty$, Algorithm \eqref{BC21} converges strongly.

On the other hand, projection and contraction methods have gained popularity for solving variational inequality problems, especially in Euclidean spaces. Algorithms employing these techniques, including the extragradient method \cite{JX22} and its variants, have shown strong convergence under monotonicity or pseudo-monotonicity of operators. For example, Cai et al. \cite{CGH14} demonstrated the effectiveness of projection and contraction methods in monotone variational inequalities with complexity analysis based on step size conditions. Dong et al. \cite{DYY19} extended these methods to infinite-dimensional Hilbert spaces, providing modified algorithms to address practical challenges. These studies highlight the potential of projection methods to effectively address variational inclusions and variational inequalities. Recently, Jia and Xu \cite{JX22} presented the following projection-like method for variational inequalities in a closed and convex set $K \subseteq \mathbb{R}^n$.\\
\textbf{Initialization:} Given $\sigma > 0$, $\eta \in (0,1)$ and $\sigma \in (0,1)$.  Define the sequence $\{u_k\}$ as 
\begin{equation}\label{JX22}
	\begin{cases}
		v_k &= P_K[u_k- \lambda_k \mathcal{B}(u_k)],\\
		\phi(u_k,\,v_k) &= (u_k - v_k) - \lambda_k (\mathcal{B} u_k - \mathcal{B} v_k), \\
		\alpha_k &= \frac{\langle u_k - v_k, \, \phi(u_k, v_k) \rangle}{\|\phi(u_k, v_k)\|^2}, \\
		u_{k+1} &= u_k - \alpha_k \phi(u_k, v_k),
	\end{cases}	
\end{equation}
where $\lambda_k = \sigma \eta^m$ and $m$ is the smallest nonnegative integer satisfying the Armijo-type condition 
\[
\sigma \eta^m \|\mathcal{B}(u_k) - \mathcal{B}(P_K[u_k - \sigma \eta^m \mathcal{B}(u_k)])\| \leq \sigma \|u_k - P_K[u_k - \sigma \eta^m \mathcal{B}(u_k)]\|.
\]
and $\mathcal{B}$ is continuous in $\mathbb{R}^n$ and quasimonotone in $K$.  The author discussed the convergence of Algorithm \eqref{JX22} without the Lipschitz continuity assumption.

However, properties such as strong monotonicity or Lipschitz continuity are restrictive and are not easily satisfied in many real-life practical problems. Consequently, reducing these assumptions is crucial for a broader applicability; see, e.g., \cite{AOM23,ABL18,ABS22}. Motivated by these developments, more precisely Algorithm \eqref{BC21} and \eqref{JX22}, our main goal in this direction is to extend the results of the existing methods to address MVIP involving a non-Lipschitz and non-co-coercive single-valued operator $\mathcal{B}$. Specifically, we develop and analyze new inertial forward-backward contraction-type algorithms to solve MVIP in real Hilbert spaces to enhance the speed of convergence. The proposed algorithm is designed to achieve fast weak and strong convergence while operating under weaker assumptions than traditional approaches. Our methods embed inertial terms to accelerate convergence without Lipschitz continuity. By addressing these challenges, we provide a broader framework for MVIP, extending and generalizing existing results in the literature. The results presented in this work not only enhance and generalize existing findings but also reveal that our algorithms are efficient and superior to other methods currently available in the literature.

The structure of the paper is organized as follows. Section~\ref{S:2} introduces essential notations, definitions, and preliminary results that form the foundation for the subsequent analysis. Section~\ref{S:3} presents an inertial-based contraction-type method and its associated parameters. We establish key properties of the algorithm and prove its weak convergence with a rate of $\mathcal{O}(1/\sqrt{k})$. In Section~\ref{S:4}, we derive a strong convergence result under suitable conditions and demonstrate that the proposed method achieves a linear convergence rate. Section~\ref{S:5} is devoted to a computational study in which the performance of the proposed algorithm is compared with several existing methods through numerical experiments.


\section{Preliminaries}\label{S:2}


\noindent 	Let $\mathcal{H}$ be a real Hilbert space with inner product $\langle \cdot,\, \cdot\rangle$ and induced norm $\|\cdot\|$. The weak convergence and strong convergence of $\{u_k\}_{k=0}^\infty$ to $x\in \mathcal{H}$ are represented by $u_k \rightharpoonup x$ and $u_k \to x$, respectively. For each $u, v, w \in \mathcal{H}$, the following is true
\begin{enumerate}
	\item [(I)] $\|u + v\|^2 \leq \|u\|^2 + 2\langle v, u + v \rangle$;
	\item [(II)] $\|\alpha u + (1 - \alpha)v\|^2 = \alpha\|u\|^2 + (1 - \alpha)\|v\|^2 - \alpha(1 - \alpha)\|u - v\|^2$, \quad $ \forall \, \alpha \in \mathbb{R}$;
	\item [(III)] $\|u-\alpha v\|^2 \geq (1-\alpha)\|u\|^2 -\alpha(1-\alpha)\|v\|^2$,  \quad $ \forall \, \alpha \geq 0$.
\end{enumerate}

In this section, we collect some necessary concepts and lemmas to prove the main results.

\begin{definition}
	A single-valued operator $\mathcal{B} : \mathcal{H} \to \mathcal{H}$ is called
	\begin{itemize}
		\item [(a)]  nonexpansive, if for all $u, v \in \mathcal{H}$,
		\begin{equation*}
			\|\mathcal{B}u - \mathcal{B}v\| \leq \|u - v\|.
		\end{equation*}
		\item [(b)] Lipschitz continuous with $L > 0$ if for all $u, v \in \mathcal{H}$
		\begin{equation*}
			\|\mathcal{B}u - \mathcal{B}v\| \leq L\|u - v\|.
		\end{equation*}
		\item [(c)] monotone, if for all $u, v \in \mathcal{H}$,
		\begin{equation*}
			\langle \mathcal{B}u - \mathcal{B}v, u - v \rangle \geq 0.
		\end{equation*}
		\item[(d)] strongly monotone if there is a constant $\beta>0$ such that for every $u, v \in \mathcal{H}$,
		$$ 
		\left\langle \mathcal{B}u-\mathcal{B}v, u-v \right\rangle \geq \beta \|u-v\|^2.
		$$
		\item [(e)] co-coercive ($\alpha$-inverse strongly monotone), if there exists an
		$\alpha>0$ such that
		$$ 
		\left\langle \mathcal{B}u-\mathcal{B}v, u-v \right\rangle \geq \alpha\|\mathcal{B}u-\mathcal{B}v\|^2.
		$$
	\end{itemize}
\end{definition} 
\begin{definition}
	A set-valued operator $\mathcal{A} : \mathcal{H} \rightrightarrows \mathcal{H}$ is called
	\begin{itemize}
		\item [(a)] monotone, if for all $u, v \in \mathcal{H}$,
		\begin{equation*}
			\langle w - z, u - v \rangle \geq 0, \quad \forall w \in \mathcal{A}u \mbox{ and } \forall z \in \mathcal{A}v.
		\end{equation*}
		\item[(b)] strongly monotone if there is a constant $\beta>0$ such that for every $u, v \in \mathcal{H}$,
		$$ \left\langle w-z, u-v \right\rangle \geq \beta \|u-v\|^2, \quad \forall u\in \mathcal{A}u \mbox{ and } \forall z\in \mathcal{A}v.$$
		
		\item [(c)] maximal monotone if it is monotone, in addition, its graph 
		$$
		G(\mathcal{A}) := \{(u, v) \in \mathcal{H} \times \mathcal{H} : v \in \mathcal{A}u\}
		$$
		is not contained in any  other graph of monotone operator. It is worth mentioning that a monotone operator $\mathcal{A}$ is maximal if and only if for $(u, v) \in \mathcal{H} \times \mathcal{H}$, 
		$$
		\langle u - w, v - z \rangle \geq 0 \quad \text{for every } (w, z) \in G(\mathcal{A}) \Rightarrow v \in \mathcal{A}u.
		$$
	\end{itemize}
\end{definition}

	\begin{lemma}\label{lm1}{\rm \cite[Corollary 20.25]{BC11}}
		Let $\mathcal{B} : \mathcal{H} \to \mathcal{H}$ be a monotone and continuous single-valued operator. Then $\mathcal{B}$ is maximal monotone.
	\end{lemma}	

		\begin{proposition}  
			Let $\mathcal{A} : \mathcal{H} \rightrightarrows \mathcal{H}$ be a maximal monotone set-valued operator and $\mathcal{B} : \mathcal{H} \to \mathcal{H}$ be a monotone and continuous single-valued operator with $D(\mathcal{A}) = \mathcal{H}$. Then, $\mathcal{A} + \mathcal{B}$ is maximal monotone.
		\end{proposition}
		\begin{proof}  
			Since $\mathcal{B}$ is a monotone and continuous single-valued operator, by Lemma \ref{lm1}, $\mathcal{B}$ is the maximal monotone. Using the maximal monotonicity of $\mathcal{A}$ and by Remark \ref{rm1}(c), $\mathcal{A} + \mathcal{B}$ is maximal monotone.
		\end{proof}
		Let $\mathcal{A} : \mathcal{H} \rightrightarrows \mathcal{H}$ be a set-valued operator. Define the resolvent operator $J_{\lambda\mathcal{A}}: \mathcal{H} \to \mathcal{H}$ associated with $\lambda > 0$ as
		$$
		J_{\lambda\mathcal{A}}(x) := (I + \lambda\mathcal{A})^{-1}(x), \quad \forall x \in \mathcal{H}.
		$$
		
		\begin{remark}\label{rm1}
			{\rm(a)} If $\mathcal{A} : \mathcal{H} \rightrightarrows \mathcal{H}$ is maximal monotone and $\lambda > 0$ then $D(J_{\lambda\mathcal{A}}) = \mathcal{H}$, the resolvent operator $J_{\lambda\mathcal{A}}$ is single-valued, nonexpansive, and firmly nonexpansive.
			\begin{itemize}
				\item[\rm (b)] Define the set of fixed points of operator $T$ as $\operatorname{Fix}(T) = \{x \in \mathcal{H} : x = Tx\}$. The fixed point of the resolvent is given by
				$$
				\mbox{Fix} (J_{\lambda\mathcal{A}})  = A^{-1}(0)=\{x \in D(\mathcal{A}) : 0 \in \mathcal{A}(x)\}.
				$$
				\item [\rm (c)] If $\mathcal{A} : \mathcal{H} \rightrightarrows \mathcal{H}$ and $\mathcal{B} : \mathcal{H} \to \mathcal{H}$ are maximal monotone then so is $\mathcal{A}+\mathcal{B}$.
			\end{itemize}
		\end{remark}
		\begin{lemma}\label{lm11}
			Let $\mathcal{H}$ be a real Hilbert space, $\mathcal{A} : \mathcal{H} \rightrightarrows \mathcal{H}$ be a maximal monotone operator and $\mathcal{B} : \mathcal{H} \to \mathcal{H}$ be a single-valued operator. and let
			\[
			T_\lambda = (I + \lambda\mathcal{A})^{-1}(I - \lambda\mathcal{B}), \quad \lambda > 0.
			\]
			Then,
			\[
			\operatorname{Fix}(T_\lambda) = (\mathcal{A} + \mathcal{B})^{-1}(0).
			\]
		\end{lemma}
		
		\begin{proof}
			By the definition of $T_\lambda$, we observe
			\[
			x = T_\lambda x \quad \Leftrightarrow \quad x = (I + \lambda\mathcal{A})^{-1}(I - \lambda\mathcal{B})x
			\quad \Leftrightarrow \quad x - \lambda\mathcal{B} x \in (I + \lambda\mathcal{A})(x).
			\]
			This implies that
			\[
			x \in (\mathcal{A} + \mathcal{B})^{-1}(0).
			\]
			Therefore,
			\[
			\operatorname{Fix}(T_\lambda) = (\mathcal{A} + \mathcal{B})^{-1}(0).
			\]
		\end{proof}
		
		\begin{definition}\label{df1}\cite{OR70}
			A sequence $\{u_k\}$ in a Hilbert space $\mathcal{H}$ is said to converge linearly to $u^* \in \mathcal{H}$ with rate $\vartheta \in [0,1)$ if there exists a constant $c > 0$ such that
			\[
			\|u_k - u^*\| \leq c \vartheta^k, \quad \forall k \in \mathbb{N}.
			\]
		\end{definition}
		
		\begin{lemma}\label{lm5}{\rm\cite[Lemma 2.39]{BC11} } Let $\Omega$ be a nonempty subset of $\mathcal{H}$ and $\{x_k\}$ be a sequence in $\mathcal{H}$ such that the following properties hold:
			\begin{itemize}
				\item[{\rm(a)}] $\forall x \in \Omega, \, \lim\limits_{k \to \infty} \|x_k - x\|$ exists;
				\item[{\rm(b)}] If a subsequence of $\{u_k\}$ converges weakly to $x$ and $x \in \Omega$.
			\end{itemize}
			Then the sequence $\{u_k\}$ converges weakly to a point in $\Omega$.
		\end{lemma}

		\begin{lemma}\label{lm8}{\rm \cite{AA01}}
			Let  $\{\alpha_k\}$, $\{\beta_k\}$ and $\{\omega_k\}$ be sequences in $[0,+\infty)$ such that
			\[
			\omega_{k+1} \leq \omega_k + \alpha_k(\omega_k - \omega_{k-1}) + \beta_k, \quad \forall k \in \mathbb{N},
			\]
			\[
			\sum_{k=1}^{\infty} \beta_k < +\infty,
			\]
			and there exists a real number $\alpha$ with $0 \leq \alpha_k \leq \alpha < 1$ for all $k \in \mathbb{N}$. Then the following hold
			\begin{itemize}
				\item[{\rm(a)}] $\sum\limits_{k=1}^{\infty} |\omega_k - \omega_{k-1}| < +\infty$, where $|t| := \max\{t, 0\}$;
				\item[{\rm(b)}] There exists $\omega^* \in [0,+\infty)$ such that $\lim\limits_{k \to \infty} \omega_k = \omega^*$.
			\end{itemize}
		\end{lemma} 
		
		\begin{lemma}\label{lm10}{\rm\cite[Lemma 2.4]{L90}}
			Let $\{\alpha_k\}$ and $\{\beta_k\}$ be sequences of nonnegative real numbers. Suppose there exists a constant $0 \leq \theta < 1$ such that
			\[
			\alpha_{k+1} \leq \theta \alpha_k + \beta_k, \quad \forall k \in \mathbb{N}.
			\]
			If $\lim\limits_{k \to \infty} \beta_k = 0$, then 
			$$
			\lim_{k \to \infty} \alpha_k = 0.
			$$
		\end{lemma}

	
	\section{Weak Convergence}\label{S:3}
	
	\noindent This section introduces an inertial forward-backwards algorithm to address inclusion problems within real Hilbert spaces. A key benefit of the proposed method is that it operates without assuming coercivity or Lipschitz continuity of the single-valued operator.
	\begin{assumption}\label{Assump:1}
		{\rm(A1)} The solution set of the MVIP is nonempty, i.e., 
		$$
		\Omega_{\mathcal{A} + \mathcal{B}} := (\mathcal{A} + \mathcal{B})^{-1}(0) \neq \emptyset.
		$$
		\begin{itemize}
			\item[(A2)] The set-valued operator $\mathcal{A} : \mathcal{H} \rightrightarrows \mathcal{H}$ is maximal monotone with $D(\mathcal{A})=\mathcal{H}$.
			\item[(A3)] The single-valued operator $\mathcal{B} : \mathcal{H} \to \mathcal{H}$ is monotone and continuous.
			\item[(A4)] The non-decreasing sequence $\{\vartheta_k\}\subseteq (0,\,\vartheta)$  satisfies
			\begin{equation}\label{inertialCondition}
				0\leq \vartheta_k \leq \vartheta_{k+1} \leq \vartheta < \frac{\mathcal{E}}{\mathcal{E}+\max\{1, \mathcal{E}\}} ,
			\end{equation}
			where $\mathcal{E}=\frac{2-\gamma}{\gamma} \frac{(1 - \sigma)^4}{(1 + \sigma)^4}$,  $0 < \gamma <2$ and $0 < \sigma < 1$.
		\end{itemize}
	\end{assumption}

	\begin{algorithm}\label{A:1}
	Inertial forward–backward contraction type  method\\
		\noindent \noindent {\bf Initializing:} Choose $ s > 0 $, $ \mu \in (0, 1) $, $ \sigma \in (0, 1) $, $\vartheta_k \in (0, 1)$ and $\gamma \in (0, 2)$. Let $ u_0, u_1 \in \mathcal{H} $. Set $ k = 0 $.
		
		\vspace{5pt}
		
		\noindent {\bf Step 1:} Define
		\[
		w_k = u_k + \vartheta_k (u_k - u_{k-1}).
		\]
		
		\noindent {\bf Step 2:} Select $\lambda_k = s \mu^{j_k}$, where $j_k$ is the smallest nonnegative integer satisfying
		\begin{equation}
			\label{3.1}
			\lambda_k \|\mathcal{B}w_k - \mathcal{B}(J_{\lambda_k\mathcal{A}}[w_k - \lambda_k \mathcal{B}w_k])\| \leq \sigma \|w_k - J_{\lambda_k\mathcal{A}}[w_k - \lambda_k \mathcal{B}w_k]\|.
		\end{equation}
		Compute
		\[
		v_k = J_{\lambda_k\mathcal{A}}[w_k - \lambda_k \mathcal{B}w_k].
		\]
		
		\noindent {\bf Step 3:} Set
		\[
		\phi(w_k, v_k) = (w_k - v_k) - \lambda_k (\mathcal{B}w_k - \mathcal{B}v_k).
		\]
		
		\noindent {\bf Step 4:} If $\phi(w_k, v_k) = 0$, stop. Otherwise, compute
		\[
		u_{k+1} = w_k - \gamma\delta_k \phi(w_k, v_k),
		\]
		where
		\begin{equation}
			\label{3.2}
			\delta_k =
			\begin{cases}
				\frac{\langle w_k - v_k, \phi(w_k, v_k) \rangle}{\|\phi(w_k, v_k)\|^2}, & \text{if } \phi(w_k, v_k) \neq 0; \\
				0, & \text{otherwise}.
			\end{cases}
		\end{equation}
		Update $k := k + 1$, and go to {\bf Step 1}.
\end{algorithm}

	\begin{remark}\label{rm2}
		The subsequent findings summarize our observations on Algorithm \ref{A:1}
		\begin{itemize}
			\item[(a)] If $\vartheta_k = 0$ and $\mathcal{A}= N_K$ a normal cone on a convex set $K \subseteq \mathcal{H}$,  i.e.,
			$$
			N_K(x) = \{  u \in \mathcal{H} : \langle u,\, x-y\rangle \leq 0, \forall y \in K \},
			$$
			then Algorithm \ref{A:1} reduces to  \cite[Algorithm YH]{JX22}.
			\item[(b)] If $\vartheta_k = 0$, in Algorithm \ref{A:1}, then we can get a solution under the Assumption \ref{Assump:1}. This approach is not allowed in the algorithms presented in \cite{ZW18,AOM23,JX22,TRCL24}.
			\item[(c)] Condition \eqref{inertialCondition} imposed on the sequence $\{\vartheta_k\}$ enhances the numerical efficiency of the Algorithm \ref{A:1} and offers a notable improvement over condition \eqref{BC21} of \cite{TC21}; see Section \ref{S:5}. Moreover, condition~\eqref{inertialCondition} is computationally more practical and easier to implement, whereas condition \eqref{BC21} of \cite{TC21} involves a higher computational cost.

			\item[(d)] In \cite[Algorithm 1]{YAS24} $\vartheta_k = \vartheta$ such that $2\mu \leq \frac{1}{1+\vartheta}$, $\mu \in (0,\,1)$
			which implies our approach, with a non-decreasing and bounded sequence ${\vartheta_k}$ defined by \eqref{inertialCondition}, offers more adaptive control, broader applicability and better numerical implementability compared to \cite{YAS24}. It enhances the algorithm's theoretical flexibility and practical stability; see Section \ref{S:5}.
			
			\item[(e)] \cite[Algorithm 3.1]{WLC24} requires double inertial steps $\{\alpha_k\}$, $\{\beta_k\}$ and $\{\theta_k\}$ that satisfy the strong conditions: $0 \leq \alpha_k \leq 1$; $0 \leq \beta_k \leq \beta_{k+1} \leq \beta < \dfrac{3 + 2\epsilon - \sqrt{8\epsilon + 17}}{2\epsilon}$; $0 < \theta < \theta_k \leq \theta_{k+1} \leq \dfrac{1}{1 + \epsilon}, \epsilon \in (1, +\infty)$ and $a_k = (1 - \theta_k)\beta_k + \theta_k\alpha_k$ is a non-decreasing sequence.
			In contrast, our approach \eqref{inertialCondition} involves only one sequence ${\vartheta_k}$, reducing the number of parameters and complexity of implementation. While \cite{WLC24} uses three coupled sequences ($\alpha_k$, $\beta_k$, $\theta_k$) with strict interdependent conditions, which increases the computational burden and the risk of misconfiguration.
		\end{itemize}
	\end{remark}
	\begin{proposition}
		If $\phi(w_k, v_k) = 0$ in Algorithm \ref{A:1}, then $v_k \in \Omega_{\mathcal{A} + \mathcal{B}}$.
	\end{proposition}
	\begin{proof}
		Indeed, from the definition of $\phi(w_k, v_k)$, one has
		$$
		\langle w_k - v_k, \phi(w_k, v_k) \rangle = \|w_k - v_k\|^2 - \lambda_k \langle w_k - v_k, \mathcal{B} w_k - \mathcal{B} v_k \rangle \geq\|w_k - v_k\|^2.
		$$
		This implies that if $\phi(w_k, v_k) = 0$, then $w_k = v_k$. It follows $v_k \in \Omega_{\mathcal{A} + \mathcal{B}}$ by means of Lemma \ref{lm11}.
	\end{proof}
	\begin{proposition}\label{pr1}
		Let $\{\lambda_k\}$ be an iterative point in {\rm Step 2} of Algorithm {\rm\ref{A:1}}. Then there must exist a nonnegative integer $j_k$ satisfying \eqref{3.1} and there exists $\lambda_{\min}>0$ such that 
		$$
		\lambda_{\min} \leq \lambda_k,  \quad \forall k \in \mathbb{N}.
		$$
	\end{proposition}
	\begin{proof}  
		The proof is along the lines of \cite[Theorem 3.4 (a)]{T00}. 
	\end{proof}

	\begin{lemma}\label{lm4} Assume that Assumptions A2 and A3 hold. Let  $ \{w_k\}$  be a sequence generated by Algorithm \ref{A:1} such that $ \lim\limits_{j \to \infty} \|w_{k_j} - J_{\lambda_{k_j}\mathcal{A}}[w_{k_j} - \lambda_k \mathcal{B}(w_{k_j})]\| = 0 $ and $ \{w_{k}\} $ converges weakly to $ w^* \in \mathcal{H}$, then $ w^* \in \Omega_{\mathcal{A}+\mathcal{B}} $.
	\end{lemma}
	\begin{proof}Let $ (v, u) \in \text{G}(\mathcal{A} + \mathcal{B}) $, that is, $ u \in (\mathcal{A} + \mathcal{B})v $. From the definition of $ v_k $ of the Algorithm \ref{A:1}, we see that $ (I - \lambda_{{k_j}} \mathcal{B}) w_{{k_j}} \in (I + \lambda_{{k_j}} \mathcal{A}) v_{{k_j}} $. This implies that
		$$
		\frac{1}{\lambda_{k}} (w_{{k_j}} - v_{{k_j}} - \lambda_{{k_j}} \mathcal{B} w_{{k_j}}) \in \mathcal{A} v_{{k_j}}.
		$$
		By the monotonicity of  $ \mathcal{A} $, we infer that
		$$
		\left\langle u - \mathcal{B} v - \frac{1}{\lambda_{k}} (w_{{k_j}} - v_{{k_j}} - \lambda_{{k_j}} \mathcal{B} w_{{k_j}}), v - v_{{k_j}} \right\rangle \geq 0.
		$$
		Combining this with the monotonicity of $ \mathcal{B} $, it follows
		\begin{align}\label{eq3a}
			\langle v - v_{{k_j}}, u \rangle &\geq \left\langle v - v_{{k_j}}, \mathcal{B} v + \frac{1}{\lambda_{{k_j}}} (w_{{k_j}} - v_{{k_j}} - \lambda_{{k_j}} \mathcal{B} w_{{k_j}}) \right\rangle\notag\\
			&= \langle v - v_{{k_j}}, \mathcal{B} v - \mathcal{B} v_{{k_j}} \rangle + \langle v - v_{{k_j}}, \mathcal{B} v_{{k_j}} - \mathcal{B} w_{{k_j}} \rangle + \frac{1}{\lambda_{{k_j}}}\langle v - v_{{k_j}},  w_{{k_j}} - v_{{k_j}} \rangle\notag\\
			&\geq \langle v - v_{{k_j}}, \mathcal{B} v_{{k_j}} - \mathcal{B} w_{{k_j}} \rangle + \frac{1}{\lambda_{{k_j}}}\langle v - v_{{k_j}}, w_{{k_j}} - v_{{k_j}} \rangle.
		\end{align}
		Moreover, by $ \lim\limits_{j \to \infty} \|w_{k_j} - J_{\lambda_{k_j}\mathcal{A}}[w_{k_j} - \lambda_{k_j} \mathcal{B}w_{k_j}]\| = 0 $, Proposition \ref{pr1} and \eqref{3.1}, we observe that 
		\begin{align*}
			\lambda_{\min} \|\mathcal{B}w_{k_j} - \mathcal{B}v_{k_j}\| \leq \lambda_{k_j} \|\mathcal{B}w_{k_j} - \mathcal{B}v_{k_j}\| \leq \sigma \|w_{k_j} - J_{\lambda_{k_j}\mathcal{A}}[w_{k_j} - \lambda_{k_j} \mathcal{B}w_{k_j}]\|.
		\end{align*}
		This implies that
		$$
		\lim\limits_{j \to \infty} \|\mathcal{B} w_{{k_j}} - \mathcal{B} v_{{k_j}}\| = 0.
		$$
		It follows with \eqref{eq3a} and $\lambda_{\min} \leq \lambda_{{k_j}}$
		$$
		\lim_{j \to \infty} \langle v - v_{k_{j}}, u \rangle = \langle v - w^*, u \rangle \geq 0.
		$$
		By using the maximal monotonicity of $ (\mathcal{A} + \mathcal{B}) $, we obtain
		$$
		0 \in (\mathcal{A} + \mathcal{B})w^* \quad \Rightarrow\quad  w^* \in \Omega_{\mathcal{A}+\mathcal{B}}.
		$$
	\end{proof}
	\begin{lemma}\label{lm2}
		Assume that Assumptions {\rm A1--A3} hold. Let $ \{u_k\}, \{v_k\} $ and $\{w_k\}$  be three sequences generated by Algorithm {\rm\ref{A:1}}. Then for any $ u^* \in \Omega_{\mathcal{A}+\mathcal{B}} $, we have
		$$
		\|u_{k+1} - u^*\|^2 \leq \|w_k - u^*\|^2 -  \gamma (2-\gamma) \frac{\langle w_k - v_k, \phi(w_k, v_k) \rangle^2}{\|\phi(w_k,v_k)\|^2},
		$$
		and
		$$
		\frac{\langle w_k - v_k, \phi(w_k, v_k) \rangle}{\|\phi(w_k, v_k)\|} \geq \frac{1 - \sigma}{1 + \sigma}\|w_k - v_k\|.
		$$
	\end{lemma}
	\begin{proof} From the definition of $ u_{k+1} $, we have
		\begin{align}\label{3.3}
			\|u_{k+1} - u^*\|^2 &= \|w_k - \gamma\delta_k \phi(w_k, v_k) - u^*\|^2 \notag\\
			&= \|w_k - u^*\|^2 - 2 \gamma \delta_k \langle w_k - u^*, \phi(w_k, v_k) \rangle + \gamma^2 \delta_{k}^2 \|\phi(w_k, v_k)\|^2.
		\end{align}
		On the other hand, we get  
		\begin{align}\label{15}
			\langle w_k - u^*, \phi(w_k, v_k) \rangle
			& = \langle w_k - v_k, \phi(w_k, v_k) \rangle + \langle v_k - u^*, \phi(w_k, v_k) \rangle\notag\\
			&= \langle w_k - v_k, \phi(w_k, v_k) \rangle + \langle v_k - u^*, w_k - v_k - \lambda_k (\mathcal{B} w_k - \mathcal{B} v_k) \rangle.
		\end{align}
		Since $ v_k = (I + \lambda_k \mathcal{A})^{-1}(I - \lambda_k \mathcal{B})w_k $, implies that $ (I - \lambda_k \mathcal{B})w_k \in (I + \lambda_k \mathcal{A})v_k $, and  
		thus, there exists a point $ q_k \in \mathcal{A}v_k $ such that  
		$$
		w_k - \lambda_k \mathcal{B}w_k = v_k + \lambda_k q_k,
		$$
		that is
		\begin{equation}\label{16}
			q_k = \frac{1}{\lambda_k}(w_k - v_k - \lambda_k \mathcal{B} w_k).
		\end{equation}
		According to $ u^* \in \Omega_{\mathcal{A}+\mathcal{B}} $, we have $ 0 \in (\mathcal{A} + \mathcal{B})u^* $ and $ q_k+\mathcal{B} v_k  \in (\mathcal{A} + \mathcal{B})v_k $. Since $ \mathcal{A} + \mathcal{B} $  
		is monotone, then it follows  
		$$
		\langle q_k+\mathcal{B} v_k, v_k - u^* \rangle \geq 0.
		$$
		Substituting \eqref{16}, it yields 
		$$
		\frac{1}{\lambda_k} \langle w_k - v_k - \lambda_k \mathcal{B} w_k + \lambda_k \mathcal{B} v_k, v_k - u^* \rangle \geq 0,
		$$
		and so,  
		\begin{equation}\label{17}
			\langle w_k - v_k - \lambda_k (\mathcal{B} w_k - \mathcal{B} v_k), v_k - u^* \rangle \geq 0.
		\end{equation}
		Combining \eqref{15} and \eqref{17}, we obtain  
		$$
		\langle w_k - u^*, \phi(w_k, v_k) \rangle \geq \langle w_k - v_k, \phi(w_k, v_k) \rangle.
		$$
		It follows with \eqref{3.2} and \eqref{3.3}
		\begin{align}\label{3.4}
			\|u_{k+1} - u^*\|^2 &\leq\|w_k - u^*\|^2 - 2 \gamma \delta_k \langle w_k - v_k, \phi(w_k, v_k) \rangle + \gamma^2 \delta_{k}^2 \|\phi(w_k, v_k)\|^2\notag\\
			&=\|w_k - u^*\|^2 -  \gamma (2-\gamma) \frac{\langle w_k - v_k, \phi(w_k, v_k) \rangle^2}{\|\phi(w_k,v_k)\|^2}.	
		\end{align}
		By definition of $ \phi(w_k, v_k) $ and \eqref{3.1}, we have  
		\begin{align}\label{3.5}
			\|\phi(w_k, v_k)\| &= \|w_k - v_k - \lambda_k (\mathcal{B}(w_k) - \mathcal{B}(v_k))\| \notag\\
			&\leq \|w_k - v_k\| + \lambda_k \|\mathcal{B}(w_k) - \mathcal{B}(v_k)\| \notag\\
			&\leq (1 + \sigma) \|w_k - v_k\|. 
		\end{align}  
		On the other hand, we deduce that  
		\begin{align}\label{3.6}
			\langle w_k - v_k, \phi(w_k, v_k) \rangle & = \langle w_k - v_k, w_k - v_k - \lambda_k (\mathcal{B}(w_k) - \mathcal{B}(v_k)) \rangle \notag\\
			&= \|w_k - v_k\|^2 - \lambda_k \langle w_k - v_k, \mathcal{B}(w_k) - \mathcal{B}(v_k) \rangle\notag\\
			&\geq \|w_k - v_k\|^2 - \lambda_k \|w_k - v_k\| \|\mathcal{B}(w_k) - \mathcal{B}(v_k)\|\notag\\ 
			&\geq \|w_k - v_k\|^2 - \sigma \|w_k - v_k\|^2 \notag\\
			&= (1 - \sigma) \|w_k - v_k\|^2.
		\end{align}  
		By virtue of \eqref{eq3.5} and \eqref{eq3.6}, we obtain
		\begin{equation*}
			\frac{\langle w_k - v_k, \phi(w_k, v_k) \rangle}{\|\phi(w_k, v_k)\|} \geq \frac{1 - \sigma}{1 + \sigma}\|w_k - v_k\|.	
		\end{equation*}
	\end{proof}
	\begin{theorem}\label{th1}
		Suppose that Assumptions  A$1$--A$3$ hold.	Then the sequence $\{u_k\}$ generated by Algorithm \ref{A:1}  converges weakly to $u^*\in \Omega_{\mathcal{A}+\mathcal{B}}$.
	\end{theorem}
	\begin{proof}
		From Lemma \ref{lm2}, we have
		$$
		\|u_{k+1} - u^*\|^2 \leq \|w_k - u^*\|^2 -  \gamma (2-\gamma) \frac{\langle w_k - v_k, \phi(w_k, v_k) \rangle^2}{\|\phi(w_k,v_k)\|^2}
		$$
		and
		$$
		\frac{\langle w_k - v_k, \phi(w_k, v_k) \rangle}{\|\phi(w_k, v_k)\|} \geq \frac{1 - \sigma}{1 + \sigma}\|w_k - v_k\|.
		$$
		The above relations imply that 
		\begin{align}\label{eq3.2}
			\|u_{k+1} - u^*\|^2 &\leq \|w_k - u^*\|^2 -  \gamma (2-\gamma) \frac{(1 - \sigma)^2}{(1 + \sigma)^2}\|w_k - v_k\|^2.
		\end{align}
		By definition of $ \phi(w_k, v_k) $ and \eqref{3.1}, we have  
		\begin{align*}
			\|\phi(w_k, v_k)\| &= \|w_k - v_k - \lambda_k (\mathcal{B}(w_k) - \mathcal{B}(v_k))\| \\
			&\geq \|w_k - v_k\| - \lambda_k \|\mathcal{B}(w_k) - \mathcal{B}(v_k)\| \\
			&\geq (1 - \sigma) \|w_k - v_k\|,
		\end{align*}  
		that is,
		$$
		\frac{\|w_k - v_k\|}{\|\phi(w_k, v_k)\|} \leq \frac{1}{1-\sigma}. 
		$$
		Therefore, it yields with \eqref{eq3.5} and \eqref{eq3.6}
		\begin{equation*}
			\frac{(1-\sigma)}{(1+\sigma)^2} \leq  \frac{\langle w_k - v_k, \phi(w_k, v_k) \rangle}{\|\phi(w_k,v_k)\|^2} \leq \frac{\|w_k - v_k\|}{\|\phi(w_k, v_k)\|} \leq \frac{1}{1-\sigma}.
		\end{equation*}
		that is,
		\begin{equation}\label{3.8}
			\frac{(1-\sigma)}{(1+\sigma)^2} \leq \delta_k  \leq \frac{1}{1-\sigma}.
		\end{equation}
		However, we have
		$$
		\|u_{k+1} - w_k\| = \gamma\delta_k\|\phi(w_k, v_k)\| \leq \gamma\delta_k(1+\sigma)\|w_k - v_k\| \leq \gamma\frac{1+\sigma}{1-\sigma}\|w_k - v_k\|,
		$$
		and so,
		\begin{equation}\label{3.7}
			\frac{1+\sigma}{1-\sigma}\|w_k - v_k\| \geq \frac{1}{\gamma}\|u_{k+1} - w_k\|.
		\end{equation}
		This together with \eqref{eq3.2} implies that
		\begin{align}\label{eq3.3}
			\|u_{k+1} - u^*\|^2 &\leq \|w_k - u^*\|^2 -  \frac{2-\gamma}{\gamma} \frac{(1 - \sigma)^4}{(1 + \sigma)^4}\|u_{k+1} - w_k\|^2\notag\\
			&=\|w_k - u^*\|^2 - \mathcal{E}\|u_{k+1} - w_k\|^2,
		\end{align}
		where $\mathcal{E} := \frac{2-\gamma}{\gamma} \frac{(1 - \sigma)^4}{(1 + \sigma)^4}$.   From the definition of $ w_k $, we get
		\begin{align}\label{eq3.4}
			\|w_k - u^*\|^2 &= \|u_k + \vartheta_k(u_k - u_{k-1}) - u^*\|^2 \notag\\
			&= \|(1 + \vartheta_k)(u_k - u^*) - \vartheta_k(u_{k-1} - u^*)\|^2 \notag\\
			&= (1 + \vartheta_k)\|u_k - u^*\|^2 - \vartheta_k\|u_{k-1} - u^*\|^2 + \vartheta_k(1 + \vartheta_k)\|u_k - u_{k-1}\|^2.	
		\end{align}
		By combining \eqref{eq3.3} and \eqref{eq3.4}, we get
		\begin{multline}\label{eq3.5}
			\|u_{k+1} - u^*\|^2 \leq (1 + \vartheta_k)\|u_k - u^*\|^2 - \vartheta_k\|u_{k-1} - u^*\|^2 + \vartheta_k(1 + \vartheta_k)\|u_k - u_{k-1}\|^2\\
			- \mathcal{E}\|u_{k+1} - w_k\|^2.
		\end{multline}  
		Further, using (III), we have
		\begin{align}\label{eq3.6}
			\|u_{k+1} - w_k\|^2 &= \|u_{k+1} - u_k - \vartheta_k(u_k - u_{k-1})\|^2 \notag\\
			&\geq (1-\vartheta_k)\|u_{k+1} - u_k\|^2 - \vartheta_k(1-\vartheta_k)\|u_k - u_{k-1}\|^2.
		\end{align}
		Applying this into \eqref{eq3.5}, we obtain
		\begin{multline}\label{eq3.18}
			\|u_{k+1} - u^*\|^2 \leq (1 + \vartheta_k)\|u_k - u^*\|^2 - \vartheta_k\|u_{k-1} - u^*\|^2 + \vartheta_k(1 + \vartheta_k)\|u_k - u_{k-1}\|^2\\
			- \mathcal{E}(1-\vartheta_k)\|u_{k+1} - u_k\|^2 + \mathcal{E}\vartheta_k(1-\vartheta_k)\|u_k - u_{k-1}\|^2.
		\end{multline}  
		It implies that
		\begin{multline}\label{eq3.7}
			\|u_{k+1} - u^*\|^2 \leq (1 + \vartheta_k)\|u_k - u^*\|^2 - \vartheta_k\|u_{k-1} - u^*\|^2 + \rho_k\|u_k - u_{k-1}\|^2\\
			- \mathcal{E}(1-\vartheta_k)\|u_{k+1} - u_k\|^2,
		\end{multline}  
		where $\rho_k= \vartheta_k(1 + \vartheta_k+\mathcal{E}(1-\vartheta_k))$. Replace the following
		$$
		\varphi_k := \|u_k - u^*\|^2 - \vartheta_k\|u_{k-1} - u^*\|^2 + \rho_k\|u_k - u_{k-1}\|^2.
		$$
		Thus, we obtain  
		\begin{multline*}
			\varphi_{k+1} - \varphi_k = \|u_{k+1} - u^*\|^2 - (1 + \vartheta_{k+1})\|u_k - u^*\|^2 + \vartheta_k\|u_{k-1} - u^*\|^2 +\rho_{k+1}\|u_{k+1} - u_k\|^2 \\-\rho_k\|u_k - u_{k-1}\|^2.
		\end{multline*}
		Since $\{\vartheta_k\}$ is nondecreasing in $(0,\,1)$. Then it yields that
		\begin{multline*}
			\varphi_{k+1} - \varphi_k = \|u_{k+1} - u^*\|^2 - (1 + \vartheta_k)\|u_k - u^*\|^2 + \vartheta_k\|u_{k-1} - u^*\|^2 + \rho_{k+1}\|u_{k+1} - u_k\|^2 \\- \rho_k\|u_k - u_{k-1}\|^2.
		\end{multline*}
		This follows with \eqref{eq3.7}
		\begin{equation}\label{eq3.81}
			\varphi_{k+1} - \varphi_k  \leq -(\mathcal{E}(1-\vartheta_k) - \rho_{k+1})\|u_{k+1} - u_k\|^2.
		\end{equation}
		One has
		\begin{align}\label{eq2.1}
			\rho_k &= \vartheta_k(1 + \vartheta_k) + \mathcal{E}\vartheta_k(1-\vartheta_k) \nonumber\\
			& \leq \vartheta_k(1+\vartheta_k +\mathcal{E}(1-\vartheta_k))\nonumber\\
			&\leq \vartheta_k(1+ \max\{1, \mathcal{E}\}).
		\end{align}
		Since $ 0 \leq \vartheta_k \leq \vartheta_{k+1} \leq \vartheta $, this implies that
		\begin{align*}
			-(\mathcal{E}(1-\vartheta_k) - \rho_{k+1}) &\leq -\mathcal{E} + \mathcal{E}\vartheta_k + \vartheta_{k+1}(1+ \max\{1, \mathcal{E}\}) \\
			& \leq -\mathcal{E} +\mathcal{E}\vartheta + \vartheta(1+ \max\{1, \mathcal{E}\}) \\
			&= -\mathcal{E} + \vartheta(\mathcal{E}+ (1+ \max\{1, \mathcal{E}\})).
		\end{align*}
		From \eqref{inertialCondition}, $\vartheta < \frac{\mathcal{E}}{\mathcal{E}+\max\{1, \mathcal{E}\}}$ this implies that $ \kappa :=\mathcal{E} - \vartheta(\mathcal{E} + (1+ \max\{1, \mathcal{E}\}))> 0 $. Hence, from \eqref{eq3.81}, we observe
		\begin{equation}\label{eq3.9} 
			\varphi_{k+1} - \varphi_k  \leq -\kappa\|u_{k+1} - u_k\|^2 \leq 0.
		\end{equation}  
		This implies that the sequence $ \{\varphi_k\} $ is nonincreasing. Since, we have  
		\begin{align*}
			\varphi_k &= \|u_k - u^*\|^2 - \vartheta_k\|u_{k-1} - u^*\|^2 + \rho_k\|u_k - u_{k-1}\|^2\\
			&\geq \|u_k - u^*\|^2 - \vartheta_k\|u_{k-1} - u^*\|^2.
		\end{align*}  
		It yields with the nonincreasing property of $ \{\varphi_k\} $
		\begin{align}\label{eq3.10}
			\|u_k - u^*\|^2 &\leq \vartheta_k\|u_{k-1} - u^*\|^2 + \varphi_k\notag\\
			&\leq \vartheta\|u_{k-1} - u^*\|^2 + \varphi_1\notag\\
			&\,\,\,\vdots\notag\\
			&\leq \vartheta^{k}\|u_0 - u^*\|^2 + \varphi_1(\vartheta_{k-1}+\cdots+1)\notag\\
			&\leq \vartheta^{k}\|u_0 - u^*\|^2 + \frac{\varphi_1}{1-\vartheta}.
		\end{align}    
		Further, we deduce that  
		$$
		\varphi_{k+1} = \|u_{k+1}- u^*\|^2 - \vartheta_{k+1}\|u_{k} - u^*\|^2 + \rho_{k+1}\|u_{k+1} - u_{k}\|^2
		\geq -\vartheta_{k+1}\|u_k - u^*\|^2.
		$$
		Using $\vartheta_{k+1} \leq \vartheta$ and  \eqref{eq3.10}, we get
		$$
		-\varphi_{k+1} \leq \vartheta_{k+1}\|u_k - u^*\|^2 \leq \vartheta\|u_k - u^*\|^2 \leq \vartheta^{k+1}\|u_0 - u^*\|^2 + \frac{\vartheta\varphi_1}{1 - \vartheta}.
		$$
		Since $ \{\varphi_k\} $ is nonincreasing and $\vartheta \in (0,\,1)$ then this follows from \eqref{eq3.9}
		\begin{align*}
			\kappa\sum_{k=1}^\infty \|u_{k+1} - u_k\|^2 \leq \varphi_1 - \varphi_{k+1} &\leq \vartheta^{n+1}\|u_0 - u^*\|^2 + \frac{\varphi_1}{1 - \vartheta}\\
			& \leq \|u_0 - u^*\|^2 + \frac{\varphi_1}{1 - \vartheta}.
		\end{align*}
		Therefore, we infer that
		\begin{equation}\label{eq3.8}
			\sum_{k=1}^\infty \|u_{k+1} - u_k\|^2 \leq \frac{1}{\kappa} \left(\|u_0 - u^*\|^2 + \frac{\varphi_1}{1 - \vartheta}\right)< +\infty.	
		\end{equation}
		This implies that
		\begin{equation}\label{eq3.11}
			\lim_{k \to \infty} \|u_{k+1} - u_k\| = 0.  
		\end{equation}
		On the other hand
		\begin{align*}
			\|u_{k+1} - w_k\|^2 &= \|u_{k+1} -u_k - \vartheta_k(u_k - u_{k-1})\|^2 \\
			&= \|u_{k+1} - u_k\|^2 + \vartheta_k^2\|u_k - u_{k-1}\|^2 - 2\vartheta_k\langle u_{k+1} - u_k, u_k - u_{k-1}\rangle.
		\end{align*}
		This together with \eqref{eq3.11}, we obtain
		\begin{equation}\label{eq3.111}
			\lim_{k \to \infty}\|u_{k+1} - w_k\| = 0.
		\end{equation}
		Again by \eqref{eq3.5}, we have
		\begin{align}
			\|u_{k+1} - u^*\|^2 &\leq (1 + \vartheta_k)\|u_k - u^*\|^2 - \vartheta_k\|u_{k-1} - u^*\|^2 + \vartheta_k(1 + \vartheta_k)\|u_k - u_{k-1}\|^2\notag\\
			&\quad - \mathcal{E}\|u_{k+1} - w_k\|^2\notag\\
			&\leq (1 + \vartheta_k)\|u_k - u^*\|^2 - \vartheta_k\|u_{k-1} - u^*\|^2 + \vartheta_k(1 + \vartheta_k)\|u_k - u_{k-1}\|^2\notag\\
			&\leq (1 + \vartheta_k)\|u_k - u^*\|^2 - \vartheta_k\|u_{k-1} - u^*\|^2 + 2\vartheta\|u_k - u_{k-1}\|^2\notag\\
			&= \|u_k - u^*\|^2 +\vartheta_k(\|u_k - u^*\|^2 - \|u_{k-1} - u^*\|^2) + 2\vartheta\|u_k - u_{k-1}\|^2\notag.
		\end{align} 
		By virtue of Lemma \ref{lm8} and \eqref{eq3.8}, there exists $l \in [0,\, \infty)$ such that
		\begin{equation}\label{eq3.121}
			\lim_{k \to \infty} \|u_k - u^*\|^2 := l.
		\end{equation} 
		Moreover, by \eqref{eq3.4}, we obtain  
		$$
		\lim_{k \to \infty} \|w_k - u^*\|^2 = l.
		$$  
		Thus the sequences $\{w_k\}, \{v_k\} $ and $\{u_k\} $ are bounded. \\
		\\
		Moreover, from \eqref{eq3.11} and \eqref{eq3.111}, we know
		$$
		\lim_{k \to \infty}\|u_k - w_k\| = \lim_{k \to \infty}\|u_k - u_{k+1}\| + \lim_{k \to \infty}\|u_{k+1} - w_k\| = 0.
		$$
		From \eqref{eq3.2}, we have 
		\begin{align}\label{eq3.13}
			\gamma(2-\gamma)\frac{(1 - \sigma)^2}{(1 + \sigma)^2}\|w_k - v_k\|^2 \leq  \|w_k - u^*\|^2 -\|u_{k+1} - u^*\|^2.
		\end{align}
		This implies that
		\begin{equation}\label{3.10}
			\lim_{k \to \infty}\|w_k - v_k\|=0.	
		\end{equation}
		Since $ \{u_k\} $ is bounded, then there exists a subsequence $ \{u_{k_j}\} $ of $ \{u_k\} $ such that $ u_{k_j} \rightharpoonup u^* \in \mathcal{H} $.  
		From $ \|u_k - w_k\| \to 0 $, we have $ w_{k_j} \rightharpoonup u^* $. By \eqref{3.10} and Lemma \ref{lm4}, we get $u^* \in \Omega_{\mathcal{A}+\mathcal{B}} $.
		Hence, by Lemma \ref{lm5}, the sequence $ \{u_k\} $ converges weakly to $ u^* \in \Omega_{\mathcal{A}+\mathcal{B}} $.
		
	\end{proof}
	\begin{remark}
		Compare with \cite{YAS24,WLC24,ZW18}, Theorem \ref{th1} is a more relaxed version due to its independency over Lipschitz continuity of the single-valued operator $\mathcal{B}$.
	\end{remark}
	
	The following result shows that Algorithm \ref{A:1} converges weakly with the non-asymptotic $\mathcal{O}(1/\sqrt{k})$ convergence rate.
	
	\begin{theorem}\label{th2}
		Assume that Assumptions {\rm A1}--{\rm A4} hold. Then the sequence $\{w_k\}$ generated by Algorithm \ref{A:1} converges weakly to a point in $\Omega_{\mathcal{A}+\mathcal{B}}$ with
		$$
		\min_{1 \leq j \leq k}\|u_j-v_j\| = \mathcal{O}\left(\frac{1}{\sqrt{k}}\right), \quad \forall k \in \mathbb{N}.
		$$
	\end{theorem}
	\begin{proof}
		From the inequalities \eqref{eq3.2} and \eqref{eq3.4}, we have
		\begin{multline}\label{eq3.91}
			\|u_{k+1} - u^*\|^2 \leq (1 + \vartheta_k)\|u_k - u^*\|^2 - \vartheta_k\|u_{k-1} - u^*\|^2 + \vartheta_k(1 + \vartheta_k)\|u_k - u_{k-1}\|^2\\
			- \zeta\|w_k - v_k\|^2.
		\end{multline}
		where $\zeta = \frac{2-\gamma}{\gamma} \frac{(1 - \sigma)^2}{(1 + \sigma)^2}$. Then \eqref{eq3.91} can be written as
		\begin{align*}
			\zeta\|w_k - v_k\|^2 &\leq  (1 + \vartheta_k)\|u_k - u^*\|^2 - \vartheta_k\|u_{k-1} - u^*\|^2 + \vartheta_k(1 + \vartheta_k)\|u_k - u_{k-1}\|^2\\
			&\qquad - \|u_{k+1} - u^*\|^2\\
			&= \|u_k - u^*\|^2 - \|u_{k+1} - u^*\|^2 + \vartheta_k(\|u_k - u^*\|^2 - \|u_{k-1} - u^*\|^2)\\
			&\qquad + \vartheta_k(1 + \vartheta_k)\|u_k - u_{k-1}\|^2.
		\end{align*}
		Let $\xi_k = \|u_k - u^*\|^2$, $\Gamma_k = \xi_k - \xi_{k-1}$ and $\varsigma_k = \vartheta_k(1 + \vartheta_k) \|u_k - u_{k-1}\|^2$, then by using $0 \leq \vartheta_k \leq \vartheta $, above inequality reduces to
		\begin{align}\label{eq3.12}
			\zeta \|w_k - v_k\|^2 &\leq \xi_k - \xi_{k+1} + \vartheta_k \Gamma_k + \varsigma_k \notag\\
			&\leq \xi_k - \xi_{k+1} + \vartheta_k |\Gamma_k| + \varsigma_k \notag\\
			&\leq \xi_k - \xi_{k+1} + \vartheta |\Gamma_k| + \varsigma_k.
		\end{align}
		In view of \eqref{eq3.8}, we have
		$
		\sum\limits_{k=1}^\infty \|u_{k+1} - u_k\|^2 < \infty.
		$
		Assume a positive constant $\mathcal{M}$ such that
		$
		\sum\limits_{k=1}^\infty \|u_{k+1} - u_k\|^2 \leq \mathcal{M}.
		$
		Then, we observe that
		\begin{align}\label{eq3.22}
			\sum_{k=1}^\infty \varsigma_k &= \sum_{k=1}^\infty \vartheta_k(1 + \vartheta_k) \|u_k - u_{k-1}\|^2 \notag\\
			&\leq \vartheta(1 + \vartheta) \sum_{k=1}^\infty \|u_k - u_{k-1}\|^2 \notag\\
			&\leq \vartheta(1 + \vartheta) \mathcal{M} := \mathcal{M}_1.
		\end{align}
		Owing to \eqref{eq3.91} with the definition of $\Gamma_k$, we get
		\begin{equation}
			\Gamma_{k+1} \leq \vartheta_k \Gamma_k + \varsigma_k \leq \vartheta_k |\Gamma_k| + \varsigma_k.
		\end{equation}
		Thus, it yields
		\begin{align}\label{eq3.1011}
			|\Gamma_{k+1}| &\leq \vartheta |\Gamma_k| + \varsigma_k\notag\\
			&\leq \vartheta^2|\Gamma_{k-1}| + \vartheta\varsigma_{k-1} + \varsigma_k\notag\\
			&\,\,\,\vdots\notag\\
			&\leq \vartheta^{k}|\Gamma_1|+ \vartheta^{k-1}\varsigma_{1}+\cdots+\vartheta\varsigma_{k-1} + \varsigma_k\notag\\
			&= \vartheta^{k}|\Gamma_1| + \sum_{j=1}^{k}\vartheta^{k-j}\varsigma_{j}.
		\end{align}    
		It follows with \eqref{eq3.22}
		\begin{align*}
			\sum_{k=1}^\infty |\Gamma_{k+1}| &\leq \vartheta(1 + \vartheta + \cdots)|\Gamma_1| + (1 + \vartheta + \cdots)\sum_{k=1}^{\infty}\varsigma_k \\
			&= \frac{\vartheta}{1 - \vartheta} |\Gamma_1| + \frac{1}{1 - \vartheta}\sum_{k=1}^{\infty}\varsigma_k\\
			&	\leq \frac{\vartheta}{1 - \vartheta} |\Gamma_1| + \frac{1}{1 - \vartheta}\mathcal{M}_1.
		\end{align*}
		Then from \eqref{eq3.12}, we obtain
		\begin{align*}
			\zeta \sum_{j=1}^{k} \|u_j - v_j\|^2 
			&\leq \xi_1 - \xi_{k+1} + \vartheta \sum_{j=1}^k |\Gamma_k| + \sum_{j=1}^k \tau_j \\
			&\leq \xi_1 + \vartheta(|\Gamma_1| + \sum_{j=1}^k |\Gamma_{j+1}|)+ \sum_{j=1}^k \tau_j\\
			&\leq \xi_1 +\vartheta|\Gamma_1| + \frac{\vartheta^2}{1 - \vartheta} |\Gamma_1| + \frac{\vartheta \mathcal{M}_1}{1 - \vartheta} + \mathcal{M}_1 \\
			&=\xi_1 + \frac{\vartheta}{1 - \vartheta} |\Gamma_1| + \frac{\mathcal{M}_1}{1 - \vartheta}.
		\end{align*}
		This implies that
		\begin{equation*}
			\sum_{j=1}^k \|u_j - v_j\|^2 \leq \left( \|u_1 - u^*\|^2 + \frac{\vartheta}{1 - \vartheta}|\|u_1 - u^*\|^2 - \|u_0 - u^*\|^2| + \frac{\mathcal{M}_1}{1 - \vartheta} \right) \frac{1}{\zeta},
		\end{equation*}
		and so,
		\begin{equation*}
			\min_{1 \leq j \leq k} \|u_j - v_j\|^2 \leq 
			\frac{ \|u_1 - u^*\|^2 + \frac{\vartheta}{1 - \vartheta} |\|u_1 - u^*\|^2 - \|u_0 - u^*\|^2| + \frac{\mathcal{M}_1}{1 - \vartheta} }{k\zeta} ,
		\end{equation*}
		and hence,
		\begin{equation*}
			\min_{1 \leq j \leq k} \|u_j - v_j\| \leq 
			\left(\frac{ \|u_1 - u^*\|^2 + \frac{\vartheta}{1 - \vartheta} |\|u_1 - u^*\|^2 - \|u_0 - u^*\|^2| + \frac{\mathcal{M}_1}{1 - \vartheta}}{k\zeta} \right)^{1/2}.
		\end{equation*}
		By Lemma \ref{lm4}, we have $v_k = w_k$ implies that $v_k$ is a solution of MVIP. This means that the error bound presented in Theorem \ref{th2} effectively characterizes the convergence rate of Algorithm \ref{A:1}.
	\end{proof}


\section{Strong convergence}\label{S:4}

\noindent In this section, we study the strong convergence and its linear convergence of the Algorithm \ref{A:1}. To this end, consider the following assumption.
\begin{assumption}\label{Assump:2}
	{\rm(B1)} The set-valued operator $\mathcal{A} : \mathcal{H} \rightrightarrows \mathcal{H}$ is maximal and strongly monotone.
	\begin{itemize}
		\item[(B2)] The single-valued operator $\mathcal{B} : \mathcal{H} \to \mathcal{H}$ is monotone and continuous.
	\end{itemize}
\end{assumption}
\begin{lemma}\label{lm3}
	Suppose that Assumptions {\rm A1}, {\rm A4} and {\rm \ref{Assump:2}} hold.  Let $ \{u_k\}, \{v_k\} $ and $ \{w_k\} $ be sequences generated by Algorithm {\rm\ref{A:1}}. Then, for any $ u^* \in \Omega_{\mathcal{A}+\mathcal{B}} $, we have
	\begin{align*}
		\|u_{k+1} - u^*\|^2 &\leq \|w_k - u^*\|^2 -\mathcal{E}\|u_{k+1} - w_k\|^2-2\mathcal{Q}\|v_k - u^*\|^2,
	\end{align*}
	where $\mathcal{E} = \frac{2-\gamma}{\gamma} \frac{(1 - \sigma)^4}{(1 + \sigma)^4}$ and $\mathcal{Q}= \gamma \lambda_{\min} \beta\frac{(1-\sigma)^2}{(1+\sigma)^2}$.
\end{lemma}
\begin{proof} 	Since 
	\[
	v_k = (I + \lambda_k \mathcal{A})^{-1}(I - \lambda_k \mathcal{B})w_k.
	\]
	This implies that 
	\[
	(I - \lambda_k \mathcal{B})w_k \in (I + \lambda_k \mathcal{A})v_k,
	\]
	hence,
	\[
	w_k - v_k - \lambda_k \mathcal{B} w_k \in \lambda_k \mathcal{A} v_k.
	\]
	On the other hand
	\[
	-\lambda_k \mathcal{B} u^* \in \lambda_k \mathcal{A} u^*.
	\]
	Since the operator \( \mathcal{A}\) is $\beta$-strongly monotone, then we have
	\[
	\langle w_k - v_k - \lambda_k \mathcal{B} w_k + \lambda_k \mathcal{B} u^*, v_k - u^* \rangle \geq \lambda_k \beta \|v_k - u^*\|^2.
	\]
	By using the monotonicity of $\mathcal{B}$, this implies that 
	\begin{align*}
		\langle w_k - v_k - \lambda_k (\mathcal{B} w_k - \mathcal{B} v_k), v_k - u^* \rangle 
		&\geq \lambda_k \beta \|v_k - u^*\|^2 + \lambda_k \langle \mathcal{B} v_k - \mathcal{B}u^*, v_k - u^* \rangle \\
		&\geq \lambda_k \beta \|v_k - u^*\|^2.
	\end{align*}
	Now, from \eqref{15}, we have
	\begin{align*}
		\langle w_k - u^*, \phi(w_k, v_k)
		&= \langle w_k - v_k, \phi(w_k, v_k) \rangle + \langle v_k - u^*, w_k - v_k - \lambda_k (\mathcal{B} w_k - \mathcal{B} v_k) \rangle\\
		&\geq \langle w_k - v_k, \phi(w_k, v_k) \rangle + \lambda_k \beta \|v_k - u^*\|^2.
	\end{align*}	
	This together with \eqref{3.3}, we obtain
	\begin{align*}
		\|u_{k+1} - u^*\|^2 &\leq 	\|w_k - u^*\|^2 - 2 \gamma \delta_k \langle w_k - u^*, \phi(w_k, v_k) \rangle + \gamma^2 \delta_{k}^2 \|\phi(w_k, v_k)\|^2\\
		& \leq \|w_k - u^*\|^2 - 2 \gamma \delta_k \langle w_k - v_k, \phi(w_k, v_k) \rangle - 2 \gamma \delta_k \lambda_k \beta \|v_k - u^*\|^2 + \gamma^2 \delta_{k}^2 \|\phi(w_k, v_k)\|^2.
	\end{align*}
	Using \eqref{3.2}, this yields that
	\begin{multline*}
		\|u_{k+1} - u^*\|^2 \leq \|w_k - u^*\|^2 -  \gamma (2-\gamma) \frac{\langle w_k - v_k, \phi(w_k, v_k) \rangle^2}{\|\phi(w_k,v_k)\|^2} - 2 \gamma \delta_k \lambda_k \beta \|v_k - u^*\|^2.
	\end{multline*}
	Similarly, by the same argument as in Lemma \ref{lm2}, we use \eqref{3.4} to obtain desired result
	$$
	\|u_{k+1} - u^*\|^2 \leq \|w_k - u^*\|^2 -  \gamma (2-\gamma) \frac{(1 - \sigma)^2}{(1 + \sigma)^2}\|w_k - v_k\|^2 - 2 \gamma \delta_k \lambda_k \beta \|v_k - u^*\|^2.
	$$ 
	
\end{proof}

\begin{theorem}\label{th3}Assume that Assumptions {\rm A1}, {\rm A4} and {\rm \ref{Assump:2}} hold. Let  $\beta < 1/\gamma\lambda_{\min}$ then the sequence $ \{u_k\} $ generated by Algorithm \ref{A:1} converges strongly to $u^* \in \Omega_{\mathcal{A}+\mathcal{B}} $.  
\end{theorem}

\begin{proof}  Indeed, 	from Lemma \ref{lm3}, we have
	$$
	\|u_{k+1} - u^*\|^2 \leq \|w_k - u^*\|^2 -  \gamma (2-\gamma) \frac{(1 - \sigma)^2}{(1 + \sigma)^2}\|w_k - v_k\|^2 - 2 \gamma \delta_k \lambda_k \beta \|v_k - u^*\|^2.
	$$ 
	According to Proposition \ref{pr1}, there exists $\lambda_{\min} > 0$ such that $\lambda_k \geq \lambda_{\min}, \forall n\in \mathbb{N}$.
	Then this together with \eqref{3.8}, we get
	\begin{align*}
		\|u_{k+1} - u^*\|^2 \leq \|w_k - u^*\|^2 -  \gamma (2-\gamma) \frac{(1 - \sigma)^2}{(1 + \sigma)^2}\|w_k - v_k\|^2 - 2 \gamma \lambda_{\min} \beta \frac{(1-\sigma)}{(1+\sigma)^2}\|v_k - u^*\|^2.
	\end{align*}
	that is,
	\begin{align}\label{eq3.14}
		\|u_{k+1} - u^*\|^2 &\leq \|w_k - u^*\|^2 -  \gamma (2-\gamma) \frac{(1 - \sigma)^2}{(1 + \sigma)^2}\|w_k - v_k\|^2 - 2 \gamma \lambda_{\min} \beta \frac{(1-\sigma)^2}{(1+\sigma)^2}\|v_k - u^*\|^2.
	\end{align}
	Thanks to the definition of $w_k$, we have
	\begin{align}\label{eq3.15}
		\|w_k - u^*\|^2 
		&= \|u_k + \vartheta_k(u_k - u_{k-1}) - u^*\|^2 \nonumber \\
		&\leq \|u_k - u^*\|^2 + \vartheta_k^2 \|u_k - u_{k-1}\|^2 + 2\vartheta_k \|u_k - u^*\| \cdot \|u_k - u_{k-1}\|.
	\end{align}
	Combining \eqref{eq3.14} and \eqref{eq3.15}  with the definition of $u_{k+1}$, we obtain
	\begin{align}\label{sd}
		\|u_{k+1} - u^*\|^2 &\leq \|u_k - u^*\|^2 + \vartheta_k^2 \|u_k - u_{k-1}\|^2 + 2\vartheta_k \|u_k - u^*\| \cdot \|u_k - u_{k-1}\| \nonumber \\
		&\qquad - \gamma (2-\gamma) \frac{(1 - \sigma)^2}{(1 + \sigma)^2}\|w_k - v_k\|^2 - 2 \gamma  \lambda_{\min} \beta \frac{(1-\sigma)^2}{(1+\sigma)^2}\|v_k - u^*\|^2 \nonumber \\
		&\leq  \|u_k - u^*\|^2 + \vartheta_k^2 \|u_k - u_{k-1}\|^2 + 2\vartheta_k \|u_k - u^*\| \cdot \|u_k - u_{k-1}\| \nonumber \\
		&\qquad -2 \gamma \lambda_{\min}\beta\frac{(1-\sigma)^2}{(1+\sigma)^2}\|v_k - u^*\|^2.
	\end{align}
	In addition,
	\begin{align*}
		\|u_k - u^*\|^2 &\leq 2\left(\|u_k - v_k\|^2 + \|v_k - u^*\|^2\right) \\
		&\leq 4\left(\|u_k - w_k\|^2 + \|v_k - w_k\|^2\right) + 2\|v_k - u^*\|^2,
	\end{align*}
	which implies that
	\begin{equation}\label{eq3.17}
		\|v_k - u^*\|^2 \geq \frac{1}{2} \|u_k - u^*\|^2 - 2\|v_k - w_k\|^2 - 2\|w_k - u_k\|^2.
	\end{equation}
	In view of \eqref{sd} and \eqref{eq3.17}, we infer that
	\begin{align*}
		\|u_{k+1} - u^*\|^2 &\leq  \|u_k - u^*\|^2 + \vartheta_k^2 \|u_k - u_{k-1}\|^2 + 2\vartheta_k \|u_k - u^*\| \cdot \|u_k - u_{k-1}\| \nonumber \\
		&\qquad -2 \gamma \lambda_{\min} \beta\frac{(1-\sigma)^2}{(1+\sigma)^2} \left(\frac{1}{2} \|u_k - u^*\|^2 - 2\|v_k - w_k\|^2 - 2\|w_k - u_k\|^2\right) \nonumber \\
		&= (1 -  \theta) \|u_k - u^*\|^2 + \vartheta_k^2 \|u_k - u_{k-1}\|^2+2\vartheta_k \|u_k - u^*\| \cdot \|u_k - u_{k-1}\| \nonumber \\
		&\qquad + 4\theta \|v_k - w_k\|^2 + 4\theta \|w_k - u_k\|^2,
	\end{align*}	
	where $\theta= \gamma \lambda_{\min} \beta\frac{(1-\sigma)^2}{(1+\sigma)^2}$. Since $\gamma\lambda_{\min}\beta < 1$, this implies that $\theta \leq 1$, and consequently $1 - \theta \in (0, 1)$. 	Given that $\{u_k\}$  is bounded, and using equations \eqref{eq3.11} and \eqref{3.10}, we obtain
	\[
	\lim_{k \to \infty} \|u_{n} - u_k\| = \lim_{k \to \infty} \vartheta_k\|u_k - u_{k-1}\| \leq \vartheta \lim_{k \to \infty} \|u_k - u_{k-1}\| = 0,
	\]
	owing to Lemma \ref{lm10}, we obtain the desired result
	\begin{equation*}
		\lim_{k \to \infty} \|u_k - u^*\|= 0.
	\end{equation*}
\end{proof}
\begin{remark}
	Theorem \ref{th3} represents one of the few strong convergence results available in the literature compare to \cite{YAS24,WLC24,TC21,TC19,TRCL24}, which does not require Lipschitz continuity of the single-valued operator $\mathcal{B}$.
\end{remark}

Further, we present the linear convergence rate of the proposed Algorithm \ref{A:1} with the following assumptions.
\begin{assumption}\label{Assump:3}
	{\rm(C1)} The solution set of the MVIP is nonempty, i.e., 
	$$
	\Omega_{\mathcal{A}+\mathcal{B}} := (\mathcal{A}+\mathcal{B})^{-1}(0) \neq \emptyset.
	$$
	\begin{itemize}
		\item[(C2)] The set valued operator $\mathcal{A} : \mathcal{H} \rightrightarrows \mathcal{H}$ is maximal and $\beta$-strongly monotone such that $\beta < 1/\gamma\lambda_{\min}$.
		\item[(C3)] The single valued operator $\mathcal{B} : \mathcal{H} \to \mathcal{H}$ is monotone and continuous.
		\item[(C4)] The sequence $\{\vartheta_k\}\subseteq (0,\,1)$ is non-decreasing such that
		\begin{equation}\label{inertialCondition2}
			0\leq \vartheta_k \leq \vartheta_{k+1} \leq \vartheta < \min\left\{\frac{\mathcal{E}}{\mathcal{E}+\max\{1, \mathcal{E}\}}, \frac{1-\tau}{\tau}\right\},
		\end{equation}
		where $\mathcal{E}=\frac{2-\gamma}{\gamma} \alpha^2$ and $\alpha=\frac{(1 - \sigma)^2}{(1 + \sigma)^2}$ and $\tau =1- \frac{1}{2} \alpha\min\{\gamma(2-\gamma),\, 2\gamma\lambda_{\min} \beta\}$.
	\end{itemize}
\end{assumption}
\begin{theorem}\label{th4}
	Assume that the Assumption {\rm\ref{Assump:3}} holds. Then the sequence $\{u_k\}$ generated by Algorithm \ref{A:1}
	converges to a solution $u^*\in \Omega_{\mathcal{A}+\mathcal{B}}$  with a linear convergence rate.
\end{theorem}
\begin{proof}
	From \eqref{eq3.14}, we observe that
	\begin{equation*}
		\|u_{k+1} - u^*\|^2 \leq \|w_k - u^*\|^2 -  \gamma (2-\gamma) \frac{(1 - \sigma)^2}{(1 + \sigma)^2}\|w_k - v_k\|^2 - 2 \gamma \lambda_{\min} \beta \frac{(1-\sigma)}{(1+\sigma)^2}\|v_k - u^*\|^2.
	\end{equation*}
	that is,
	\begin{align}\label{3.12}
		\|u_{k+1} - u^*\|^2 &\leq \|w_k - u^*\|^2 -  \gamma (2-\gamma) \frac{(1 - \sigma)^2}{(1 + \sigma)^2}\|w_k - v_k\|^2 - 2 \gamma \lambda_{\min} \beta \frac{(1-\sigma)^2}{(1+\sigma)^2}\|v_k - u^*\|^2\notag\\
		&= \|w_k - u^*\|^2 -  \gamma (2-\gamma) \alpha\|w_k - v_k\|^2 - 2 \gamma \lambda_{\min} \beta \alpha\|v_k - u^*\|^2\notag\\
		&\leq \|w_k - u^*\|^2 - \frac{1}{2} \alpha\min\{\gamma(2-\gamma),\, 2\gamma\lambda_{\min} \beta\} \|w_k - u^*\|^2\notag\\
		&=\tau\|w_k - u^*\|^2.
	\end{align}
	where $\tau = 1- \frac{1}{2} \alpha\min\{\gamma(2-\gamma),\, 2\gamma\lambda_{\min} \beta\} \in (0,\,1)$. Since  $1 \leq 1+\vartheta_k$, this implies that $\vartheta_k \leq \vartheta_k(1+\vartheta_k)\tau$ for all $k\in \mathbb{N}$. Therefore, by using $\vartheta_{k-1} \leq \vartheta_k$ together with  \eqref{eq3.4} and \eqref{3.12}, we infer that
	\begin{align}\label{3.13}
		\|u_{k+1} - u^*\|^2 &\leq \tau\left((1 + \vartheta_k)\|u_k - u^*\|^2  + \vartheta_k(1 + \vartheta_k)\|u_k - u_{k-1}\|^2\right)\notag\\
		&\leq \tau(1+\vartheta_k)\left( \|u_k - u^*\|^2  + \vartheta_k\|u_k - u_{k-1}\|^2\right)\notag\\
		&\leq \tau(1+\vartheta_k)\big( \tau(1+\vartheta_{k-1})\left(\|u_{k-1} - u^*\|^2  + \vartheta_{k-1}\|u_{k-1} - u_{k-2}\|^2\right)\notag\\
		&\qquad + \vartheta_k\|u_k - u_{k-1}\|^2\big)\notag\\
		&\leq \tau(1+\vartheta_k)\big( \tau(1+\vartheta_k)\left(\|u_{k-1} - u^*\|^2  + \vartheta_k\|u_{k-1} - u_{k-2}\|^2\right)\notag\\
		&\qquad + \vartheta_k(1+\vartheta_k)\tau\|u_k - u_{k-1}\|^2\big)\notag\\
		&\leq \tau_2(1+\vartheta_k)^2\big(\|u_{k-1} - u^*\|^2  + \vartheta_k\|u_{k-1} - u_{k-2}\|^2+ \vartheta_k\|u_k - u_{k-1}\|^2\big)\notag\\
		&\leq \tau_2(1+\vartheta_k)^2\big(\|u_{k-1} - u^*\|^2  + \vartheta_k\|u_{k-1} - u_{k-2}\|^2 + \vartheta_k\|u_k - u_{k-1}\|^2\big)\notag\\
		&\,\,\,\vdots\notag\\
		&\leq \tau^{k}(1+\vartheta_k)^{k}\big(\|u_1 - u^*\|^2  + \vartheta_k\sum_{j=1}^{k}\|u_j - u_{j-1}\|^2 \big).
	\end{align}
	Since $\vartheta_k \leq \vartheta$ and from \eqref{eq3.8} there exists $\mathcal{M}>0$ such that
	\begin{align}
		\|u_{k+1} - u^*\|^2			&\leq \tau^{k}(1+\vartheta)^{k}\big(\|u_1 - u^*\|^2  + \vartheta \mathcal{M} \big)\notag\\
		&= (1+\vartheta - \tau(1+\vartheta))^k\big(\|u_1 - u^*\|^2  + \vartheta \mathcal{M} \big)\notag\\
		&= (1 - \tau(1+\vartheta)+\vartheta)^k\big(\|u_1 - u^*\|^2  + \vartheta \mathcal{M} \big).
	\end{align}
	Since  $\vartheta \leq \tau(1+\vartheta)$, this follows $1 - \tau(1+\vartheta)+\vartheta \in (0,\, 1)$. Thus, we obtain the desired result by the definition \ref{df1}. 
\end{proof}
\begin{remark}
	Theorem~\ref{th4} presents a significant improvement over the result established in \cite[Theorem 4.2]{WLC24}, where the authors assumed that $\mathcal{A} : \mathcal{H} \rightrightarrows \mathcal{H}$ is a maximally and $r$-strongly monotone operator, and $\mathcal{B} : \mathcal{H} \to \mathcal{H}$ is a monotone and $L$-Lipschitz continuous mapping. Additionally, their framework required the inertial  sequences $\{ \alpha_k \}$, $\{ \beta_k \}$, and $\{ \theta_k \}$ subject to the following parameter constraints:
	\[
	\hat{\lambda} := \min\left\{\frac{\mu}{L}, \lambda_1\right\}, \quad 
	\tau := 1 - \frac{1}{2} \min\left\{1 - \mu, 2 \hat{\lambda} r \right\} \in \left(\frac{1}{2}, 1\right),
	\]
	and
	\begin{itemize}
		\item[(a)] $0 \leq \beta_k \leq \beta < \dfrac{1}{2}\left(\dfrac{1}{\tau} - 1\right),$
		
		\item[(b)] $0 \leq \alpha_k \leq \alpha < \dfrac{1 - \tau}{\tau},$
		
		\item[(c)] 
		\begin{multline*}
			\max\left\{\dfrac{1 - \beta}{1 + \alpha - \beta}, \dfrac{\beta}{1 + \beta - \tau(1 + \alpha)}\right\} < \theta \leq \theta_{k-1} \leq \theta_k \\ \leq 
			\frac{-1 - \beta + \sqrt{(1 + \beta)^2 - 4\left(\frac{1}{\tau} - 1 - 2\beta\right)(\beta - 1)}}{2\left(\frac{1}{\tau} - 1 - 2\beta\right)}.
		\end{multline*}
	\end{itemize}
	Alternatively, our proposed analysis eliminates the need for such restrictive and intertwined conditions by introducing more relaxed, practical, and easily verifiable assumptions, which are comprehensively presented in Assumption~\ref{Assump:3}.
\end{remark}

	\section{Computational Experiment}\label{S:5}
\noindent This section presents several examples where the operator $\mathcal{B} = \nabla f$ is monotone but neither Lipschitz continuous nor co-coercive. Such examples are essential to highlight the applicability of our proposed inertial-based contraction-type method. We illustrate numerical experiments based on classical benchmark problems and evaluate the performance of the proposed Algorithm \ref{A:1} (denoted by IFB)  in comparison with several state-of-the-art methods, including \cite[TC]{TC21}, \cite[YAS]{YAS24}, \cite[ZW]{ZW18}, \cite[TRCL]{TRCL24} and \cite[WLC]{WLC24}. All the numerical experiments were conducted using MATLAB version R2021b.

We adopted the parameter settings recommended in the original papers for a fair comparison. If these settings lead to suboptimal performance, we apply minor adjustments, ensuring consistency with our method's tuning principles. The parameter configurations used for the competing algorithms are summarized below.

\begin{itemize}
	\item [{(\rm a)}] IFB: set $s= 1$, $\mu= 0.5$, $\sigma = 0.9$, $\gamma= 1.9$, $\vartheta_k = \frac{\vartheta\sqrt{k}}{k+5}$ and $\vartheta = 0.99\frac{\mathcal{E}}{\mathcal{E}+\max\{1, \mathcal{E}\}}$.
	
	\item [{(\rm b)}] TC: set $\delta= 2$, $l= 0.5$, $\mu = 0.5$, $\gamma= 1$, $\alpha_k = \frac{1}{k+1}$,
	$\beta_k = 0.5(1- \alpha_k)$, $f(x)= 0.5x$ and $\varepsilon = \frac{100}{(k+1)^2}$ and $\vartheta = 0.5$.
	
	\item [{(\rm c)}] YAS: set $\mu = 0.0001$, $\lambda_{-1} = \lambda_0 =
	0.001$ and $\vartheta= 10$.
	
	\item [{\rm d)}] ZW: set $\lambda_k = \frac{k}{1+k}$ and $c=0.5$.
	
	\item [{(\rm e)}] TRCL: set  $\delta= 5$, $l= 0.4$, $\mu = 0.4$, $\gamma= 0.6$, $\alpha_k = 1/(k+ 1)$,
	$\beta_k = \frac{1}{2}$, $\lambda_k = \frac{1}{\|C\|^2}$, $\gamma_k = 1- \alpha_k - \beta_k$ and  $f(x)= 0.5x$.
	
	\item [{(\rm f)}] WLC: set $
	\mu = 0.9,\quad \alpha_k = 1 - \frac{1}{10k}, \quad \beta_k = 0.1 - \frac{1}{1000 + k}, \quad \vartheta_k = 0.45 - \frac{1}{1000 + k},
	\lambda_1 = 0.1, \quad \mu_k = \frac{1}{k^2}, \quad p_k = \frac{1}{k^2}
	$
	
\end{itemize}

\subsection{Signal Recovery by Compressed Sensing}

In this section, we present numerical simulations based on Algorithm \ref{A:1} to demonstrate its effectiveness in signal reconstruction using compressed sensing, a framework for acquiring and reconstructing sparse signals from limited measurements. Compressed sensing addresses the recovery of signals from an underdetermined linear system, given by
\begin{equation}
	v = Cu + \varepsilon, \tag{30}
\end{equation}
where $u \in \mathbb{R}^d$ is the original signal vector with $m$ nonzero entries, $v \in \mathbb{R}^m$ is the vector of noisy observations, $\varepsilon$ represents the noise, and $C : \mathbb{R}^d \to \mathbb{R}^m$ is a known linear measurement operator with $m << d$.

Recovering the signal $u$ from the observations $v$ can be reformulated as solving a regularized optimization problem, often expressed as a variation of the LASSO problem
\begin{equation}
	\min_{u \in \mathbb{R}^d} \ \frac{1}{4} \|Cu-v\|_2^4 + \rho \|u\|_1, \tag{31}
\end{equation}
where $\rho > 0$ is a regularization parameter that promotes sparsity in the solution.

For the numerical experiments, the actual signal $u \in \mathbb{R}^d$ is generated by selecting $l$ nonzero components drawn uniformly from the interval $[-2, 2]$, while the remaining entries are set to zero. The sensing matrix $C \in \mathbb{R}^{m \times d}$ is sampled from a Gaussian distribution with zero mean and unit variance. The measurements $v$ are then corrupted with Gaussian noise adjusted to achieve a signal-to-noise ratio (SNR) of 40 dB. The initial guess $u_0$ for the iterative algorithm is selected randomly. To evaluate the performance of the recovery process, the mean squared error (MSE) at iteration $k$ is computed as
\[
E_k = \|u_k - u^*\|^2 \leq 10^{-3},
\]
where $u_k$ is the approximation at the $k$-th iteration, and $u^*$ denotes the recovered signal.

The objective function comprises two parts \( f : \mathbb{R}^d \to \mathbb{R}\) and \( g : \mathbb{R}^d \to \mathbb{R}\) defined as
\[
f(u) := \frac{1}{4} \|C u - v\|_2^4\, \mbox{ and }\, g(u) := \rho\|u\|_1,
\]
where \( \| \cdot \|_2 \) is the standard Euclidean norm, \( \| \cdot \|_1 \) denotes the \( \ell_1 \)-norm used for regularization and $C : \mathbb{R}^d \to \mathbb{R}^m$ be a bounded linear operator. The gradient of the smooth function \( f:  \mathbb{R}^d \to \mathbb{R}\) denoted as \( \mathcal{B} := \nabla f \) and the subdifferential of the nonsmooth regularization function $g$  denoted as \( \mathcal{A} := \partial g \), are given by 
$$
\mathcal{B}=\nabla f(u) = \|Cu-v\|^2C^T(Cu - v)
$$
and
\[
\mathcal{A}=\partial \|u\|_1 = \left\{ w \in \mathbb{R}^n : \, u_i = \mathrm{sgn}(w_i) \text{ if } w_i \neq 0, w_i \in [-1, 1] \text{ if } u_i = 0 \text{ for all } i = 1, 2, \dots, k \right\},
\]
where  $\text{sgn}$ is the signum function, defined as
\[
\mathrm{sgn}(u_i) =
\begin{cases}
	1, & \text{if } u_i > 0, \\
	0, & \text{if } u_i = 0, \\
	-1, & \text{if }  u_i < 0.
\end{cases}
\]
Note that the function $f$ is convex and differentiable but not Lipschitz-continuous. 	Under this formulation, the resolvent operator $J_{\lambda\mathcal{A}}:  \mathbb{R}^d \to  \mathbb{R}^d$ is given by the proximal mapping as follows 
$$
J_{\lambda \mathcal{A}}(u) =\mbox{prox}_{\lambda g}= \operatorname{sgn}(u_i) \cdot \max\{0, |u_i| - \lambda \rho \}, \quad \forall i=1, 2, \ldots, k
$$
where \( \lambda > 0 \) is a chosen step size parameter.
\begin{figure}[H]
	\centering
	\begin{subfigure}[b]{\textwidth}
		\includegraphics[width=\textwidth]{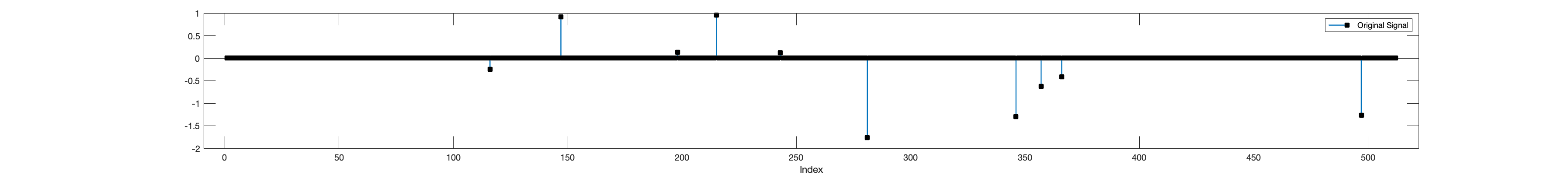}
		\caption{Original Signal ($d=512,\, m=256,\, l=10$ spikes), results shown in Table \ref{Table1}}
	\end{subfigure}
	\begin{subfigure}[b]{\textwidth}
		\includegraphics[width=\textwidth]{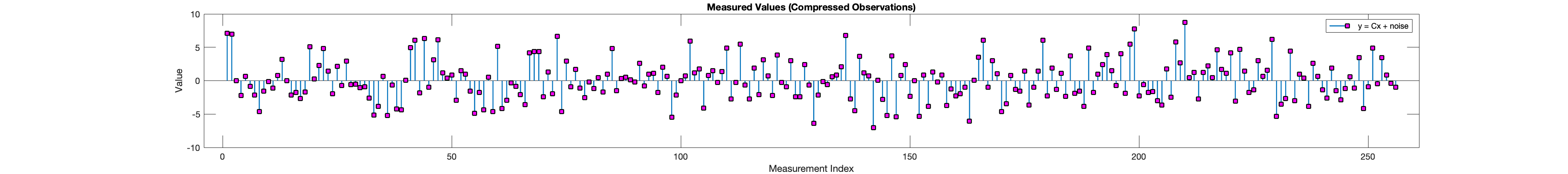}
		\caption{Measured values with SNR=40 dB}
	\end{subfigure}
	\begin{subfigure}[b]{\textwidth}
		\includegraphics[width=\textwidth]{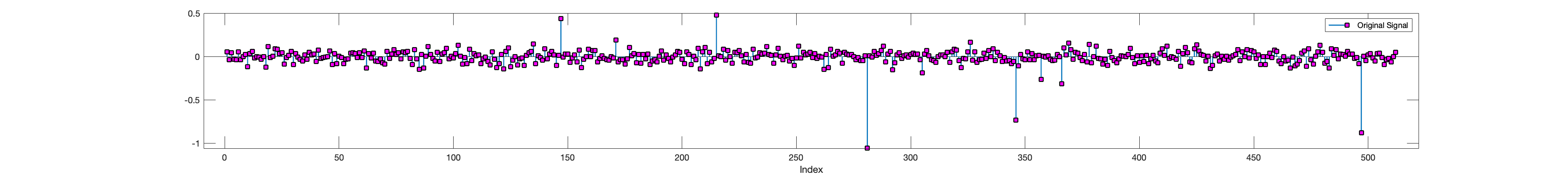}
		\caption{Recovered signal by IFB with MSE = 6.56e-03}
	\end{subfigure}
	\begin{subfigure}[b]{\textwidth}
		\includegraphics[width=\textwidth]{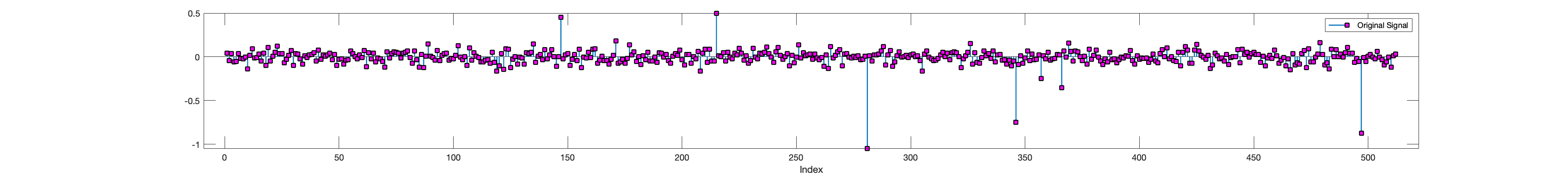}
		\caption{Recovered signal by TC with MSE  = 1.68e+01}
	\end{subfigure}
	\begin{subfigure}[b]{\textwidth}
		\includegraphics[width=\textwidth]{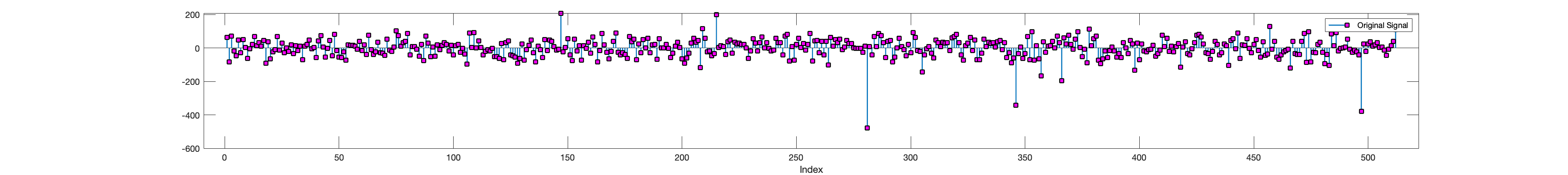}
		\caption{Recovered signal by YAS with MSE = 3.42e+03}
	\end{subfigure}
	\begin{subfigure}[b]{\textwidth}
		\includegraphics[width=\textwidth]{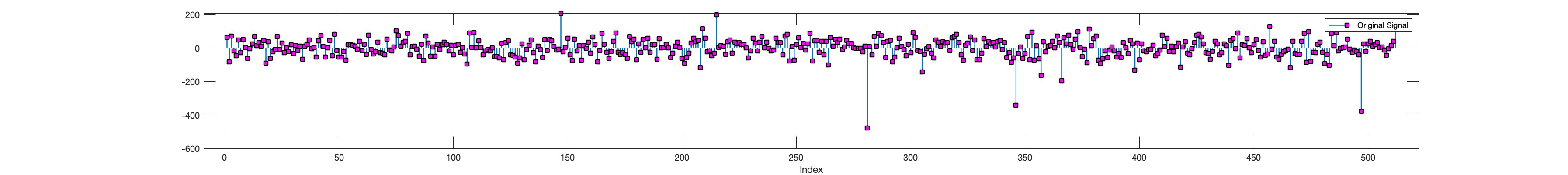}
		\caption{Recovered signal by ZW with MSE = 3.40e+03}
	\end{subfigure}
	\begin{subfigure}[b]{\textwidth}
		\includegraphics[width=\textwidth]{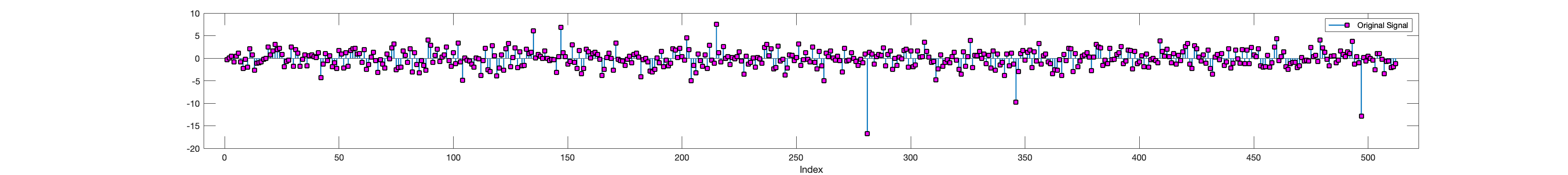}
		\caption{Recovered signal by TRCL with MSE =4.03e+01}
	\end{subfigure}
	\begin{subfigure}[b]{\textwidth}
		\includegraphics[width=\textwidth]{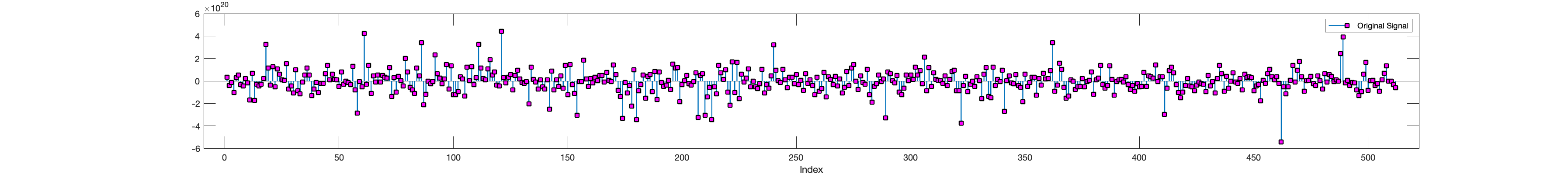}
		\caption{Recovered signal by WLC with MSE = 1.10e+40}
	\end{subfigure}
	\caption{From top to bottom: original signal, measured values, recovered signal by  IFB,
		TC, YAS, ZW, TRCL and WLC when$l=10,\, d=512,\, m=256$, respectively}
	\label{fig1}
\end{figure}

\begin{figure}[H]
	\centering
	\begin{subfigure}[b]{1\linewidth}
		\includegraphics[width=\linewidth]{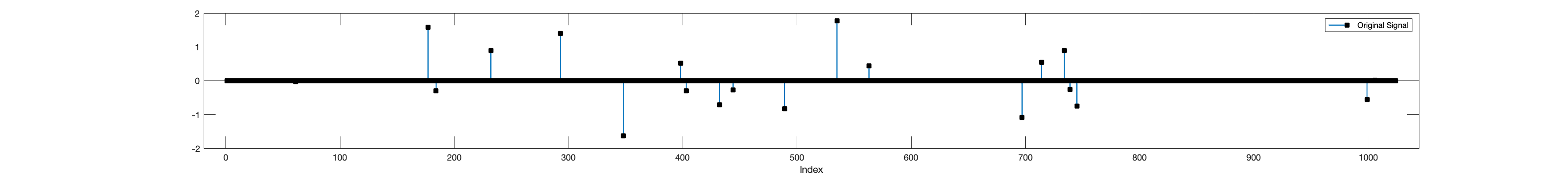}
		\caption{Original Signal ($d=1024,\, m=512,\, l=20$ spikes), results shown in Table \ref{Table1}}
	\end{subfigure}
	\begin{subfigure}[b]{\textwidth}
		\includegraphics[width=\textwidth]{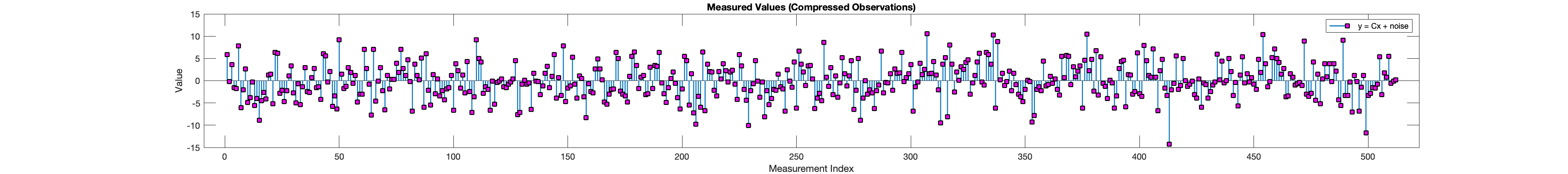}
		\caption{Measured values with SNR=40 dB}
	\end{subfigure}
	\begin{subfigure}[b]{1\linewidth}
		\includegraphics[width=\linewidth]{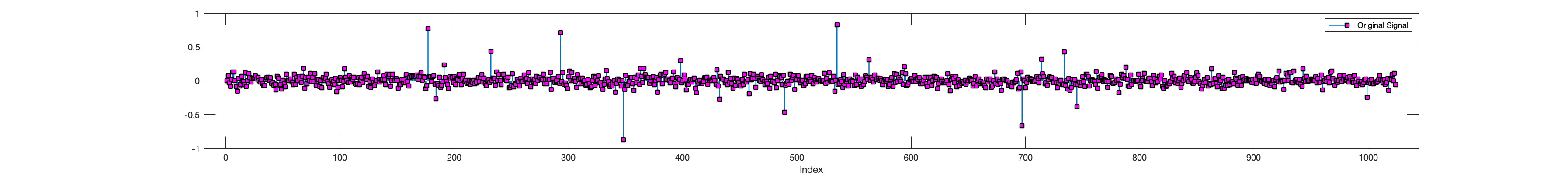}
		\caption{Recovered signal by IFB with MSE = 7.44e-03}
	\end{subfigure}
	\begin{subfigure}[b]{1\linewidth}
		\includegraphics[width=\linewidth]{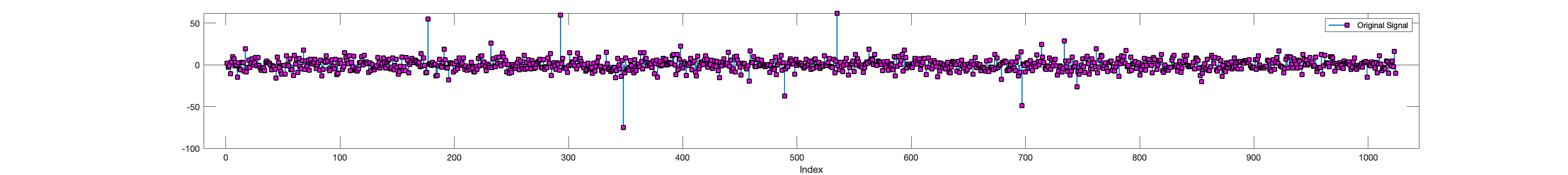}
		\caption{Recovered signal by TC with MSE = 6.00e+01}
	\end{subfigure}
	\begin{subfigure}[b]{1\linewidth}
		\includegraphics[width=\linewidth]{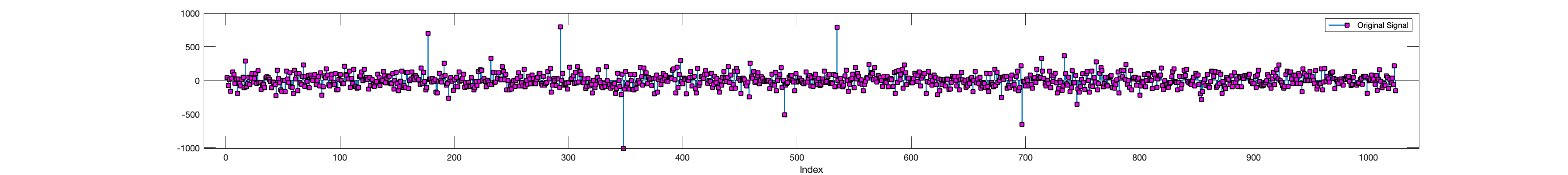}
		\caption{Recovered signal by YAS with MSE =1.22e+04}
	\end{subfigure}
	\begin{subfigure}[b]{1\linewidth}
		\includegraphics[width=\linewidth]{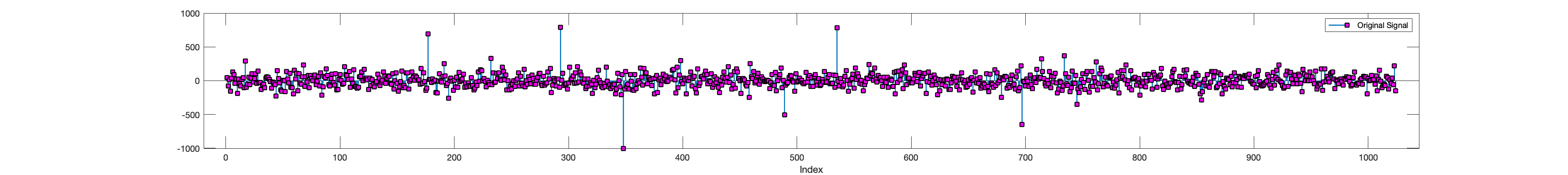}
		\caption{Recovered signal by ZW with MSE =11.20e+04}
	\end{subfigure}
	\begin{subfigure}[b]{1\linewidth}
		\includegraphics[width=\linewidth]{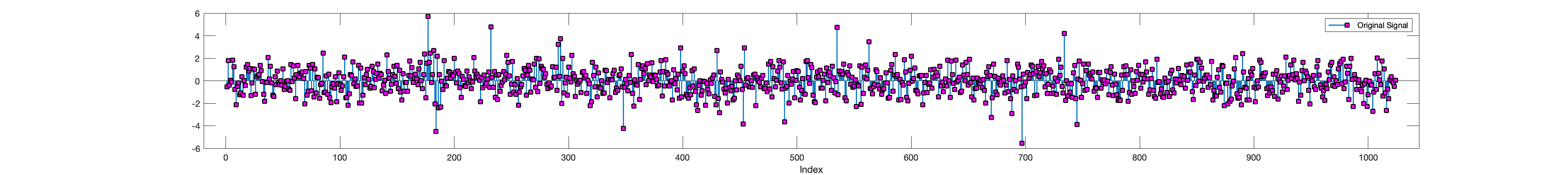}
		\caption{Recovered signal by TRCL with MSE =1.22e+01}
	\end{subfigure}
	\begin{subfigure}[b]{1\linewidth}
		\includegraphics[width=\linewidth]{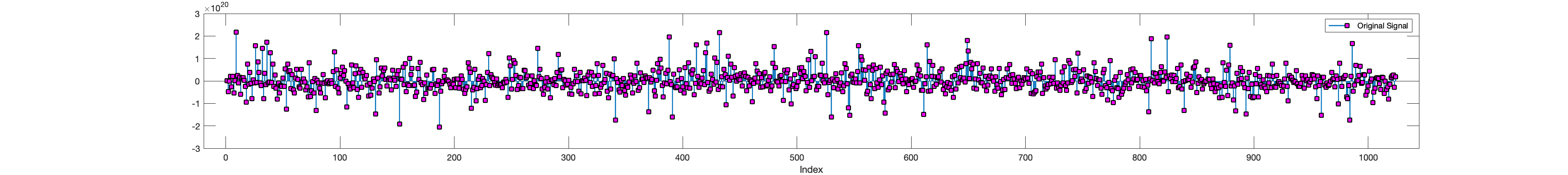}
		\caption{Recovered signal by WLC with MSE =2.55e+39}
	\end{subfigure}
	\caption{From top to bottom: original signal, measured values, recovered signal by  IFB,
		TC, YAS, ZW, TRCL and WLC when $l=20,\, d=1024,\, m=512$, respectively}
	\label{fig2}
\end{figure}
Figures \ref{fig1} and \ref{fig2} illustrate the effectiveness of our algorithm IFB in reconstructing the original signal within the compressed sensing framework. A notable strength of this approach is its ability to solve the signal recovery problem using Algorithm~\ref{A:1} without relying on the Lipschitz continuity assumption. Furthermore, the numerical experiments demonstrate that the proposed method achieves higher accuracy and greater reliability than the techniques discussed in \cite{TC21,TRCL24,YAS24,ZW18,WLC24}.
\begin{table}[!]
	\centering
	\begin{minipage}{\textwidth}
		\rule{\textwidth}{1pt}
		\begin{tabular*}{\textwidth}{@{\extracolsep{\fill}}lcccccc@{}}\\
			\multirow{1}{*}{\textbf{Sparse signals}} & \multicolumn{3}{c}{$l=10,\, d=512,\, m=256$} & \multicolumn{3}{c}{$l=20,\, d=1024,\, m=512$} \\
			\cline{2-4} \cline{5-7}\\
			& Iter. & CPU(s) & MSE & Iter. & CPU(s) & MSE \\
			\hline
			&&& &&           &       \\
			IFB & 15  & 0.000194 & 6.56e-03 & 18  & 0.000259 & 7.44e-03 \\
			&&& &&           &       \\
			TC           & 118 & 0.003420 & 1.68e+01 & 213 & 0.004150 & 6.00e+01 \\
			&&& &&           &       \\
			YAS          & 114 & 0.006541 & 3.42e+03 & 204 & 0.006649 & 1.22e+04 \\
			&&& &&           &       \\
			ZW           & 138 & 0.007621 & 3.40e+03 & 230 & 0.009774 & 1.12e+05 \\
			&&& &&           &       \\
			TRCL         & 122 & 0.006198 & 4.03e+01 & 300 & 0.005570 & 1.22e+01 \\
			&&& &&           &       \\
			WLC          & 100 & 0.048735 & 1.10e+40 & 289 & 0.072011 & 2.55e+39 \\
		\end{tabular*}
		\rule{\textwidth}{1pt}
	\end{minipage}
	\vspace{\abovecaptionskip}
	\caption{Performance comparison for various sparse signals}
	\label{Table1}
\end{table}
\begin{figure}[H]
	\centering
	\begin{subfigure}[b]{0.49\linewidth}
		\includegraphics[width=\linewidth]{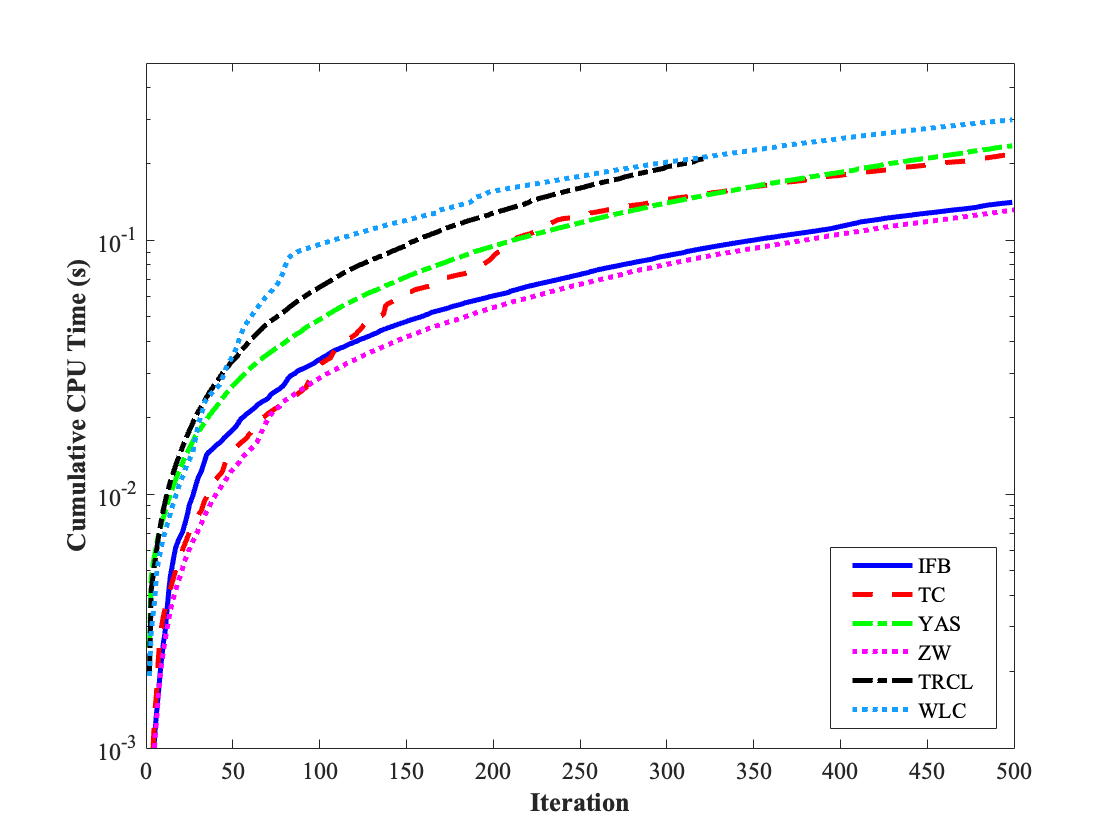}
		\caption{$d=512,\, m=256,\, l=10$ spikes}
	\end{subfigure}
	\begin{subfigure}[b]{0.49\linewidth}
		\includegraphics[width=\linewidth]{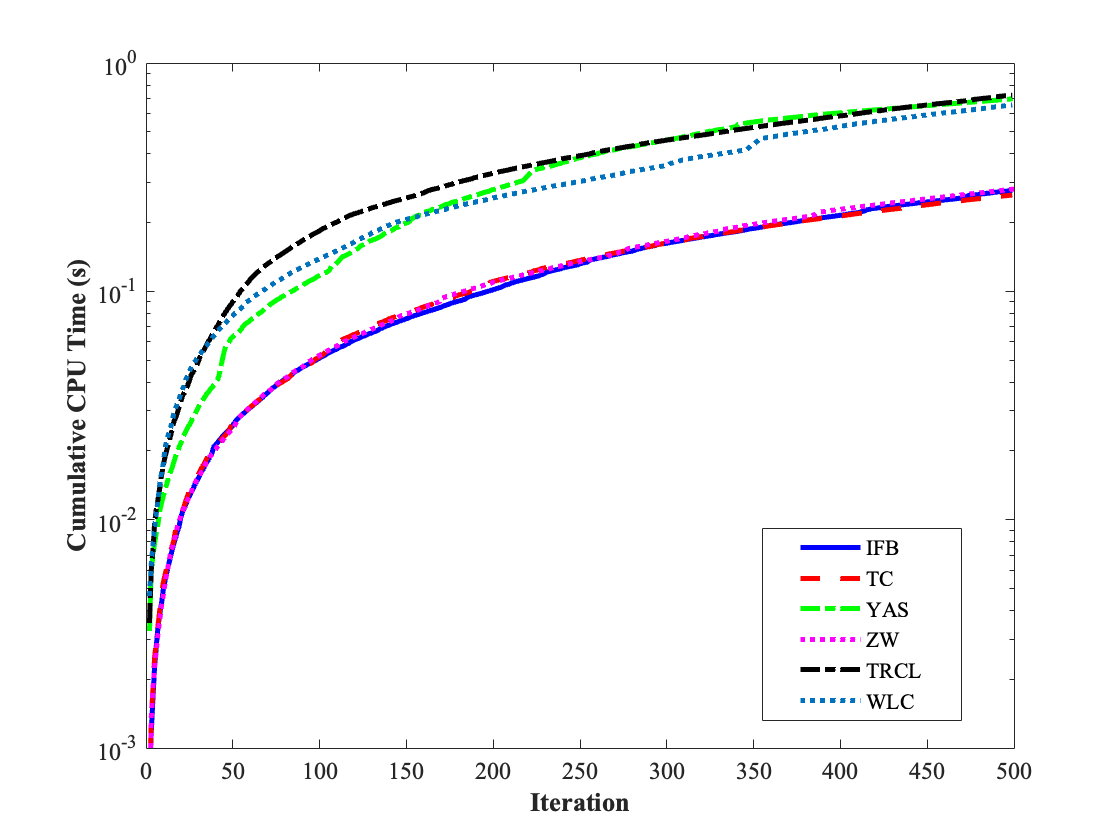}
		\caption{$d=1024,\, m=512,\, l=20$ spikes}
	\end{subfigure}
	\caption{Variation of CPU time of Figure \ref{fig1} and \ref{fig2}}
	\label{fig3}
\end{figure}		

\begin{example}\label{ex2}
	Assume a minimization problem  
	\begin{equation}\label{min2}
		\min_{u \in \mathbb{R}^d} \frac{1}{2} \|Qu - q\|_{2}^2 + \mu \sum_{i=1}^n |u_i|^\alpha+ \rho \|u\|_1.
	\end{equation}
	where, \( Q = \{p_1, p_2, \ldots, p_i\} \), where \( i = 1, 2, \ldots, m \), and each \( p_i \in \mathbb{R}^d \). The set \( q \) consists of \( m \) real values (outcomes), i.e., \( q = \{q_1, q_2, \ldots, q_i\} \) for \( i = 1, 2, \ldots, m \). The parameter \( \rho>0 \) is the sparsity controlling parameter, and \( \| \cdot \|_2 \) denotes the Euclidean norm. The nonsmooth \( \ell_1 \)-norm  \( \| \cdot \|_1 \) promotes sparsity by selecting only those attributes. The nonconvex $\sum_{i=1}^n |w_i|^\alpha$	term enhances sparsity recovery beyond convex methods like LASSO and reduces measurement requirements and improves resolution in image recovery problems, see \cite{C07,FR13}. Set the functions
	\[
	f(u) = \frac{1}{2} \|Qu - q\|^2 + \mu \sum_{i=1}^n |u_i|^\alpha, \quad \alpha \in (1, 2)
	\mbox{ and }
	g(u) = \rho \|u\|_1.
	\]
	Then
	\[
	\nabla f(u) = Q^T(Qu - q) + \mu \alpha \cdot \text{sgn}(u_i) |u_i|^{\alpha - 1}.
	\]
	Note that $\nabla f$ is not Lipschitz continuous. Indeed, consider the scalar function $\psi(u) = |u|^\alpha$. Its derivative is
	\[
	\psi'(u) = \alpha \cdot \text{sgn}(u) |u|^{\alpha - 1},
	\]
	which becomes unbounded as $u \to 0$ for $\alpha \in (1, 2)$. Hence, $\psi'$ is not Lipschitz near zero; therefore, $\nabla f$ is not Lipschitz continuous. Due to non-Lipschitz gradient $\nabla f$, classical optimization techniques may not converge. Hence, the proposed Algorithm \ref{A:1} is suitable for the minimization problem  \eqref{min2}. We use \( E_k = \|u_{k+1} - u_k\| \) to calculate the iteration's accuracy. The stopping criterion is given by
	\[
	E_k \leq 10^{-12}.
	\]
	\begin{table}[!]
		\centering
		\begin{minipage}{\textwidth}
			\rule{\textwidth}{1pt}
			\begin{tabular*}{\textwidth}{@{\extracolsep{\fill}}lcccccc@{}}\\
				\multirow{1}{*}{\textbf{Sparse signals}} & \multicolumn{3}{c}{$d=512,\, m=256$} & \multicolumn{3}{c}{$d=1024,\, m=512$} \\
				\cline{2-4} \cline{5-7}\\
				& Iter. & CPU(s) & Error & Iter. & CPU(s) & Error \\
				\hline
				&&& &&           &        \\
				IFB & 11  & 0.0035 & 5.87e-12 & 12  & 0.0040842 & 3.51e-12 \\
				&&& &&           &        \\
				TC           & 250 & 0.08512 & 5.32e-03 & 250 & 0.07641 & 5.36e-03 \\
				&&& &&           &        \\
				YAS          & 250 & 0.08612 & 2.54e-05 & 250 & 0.08865 & 3.33e-04 \\
				&&& &&           &        \\
				ZW           & 250 & 0.09651 & 2.54e-03 & 250 & 0.09153 & 3.98e-03 \\
				&&& &&           &        \\
				TRCL         & 250 & 0.16901 & 1.43e-05 & 250 & 0.15809 & 4.52e-05 \\
				&&& &&           &        \\
				WLC          & 51 & 0.18764 & 2.81e-12 & 53 & 0.17129 & 3.12e-12 \\
			\end{tabular*}
			\rule{\textwidth}{1pt}
		\end{minipage}
		\vspace{\abovecaptionskip}
		\caption{Performance comparison for Example \ref{ex2}}
		\label{Table2}
	\end{table}
	\begin{figure}[H]
		\centering
		\begin{subfigure}[b]{0.49\linewidth}
			\includegraphics[width=\linewidth]{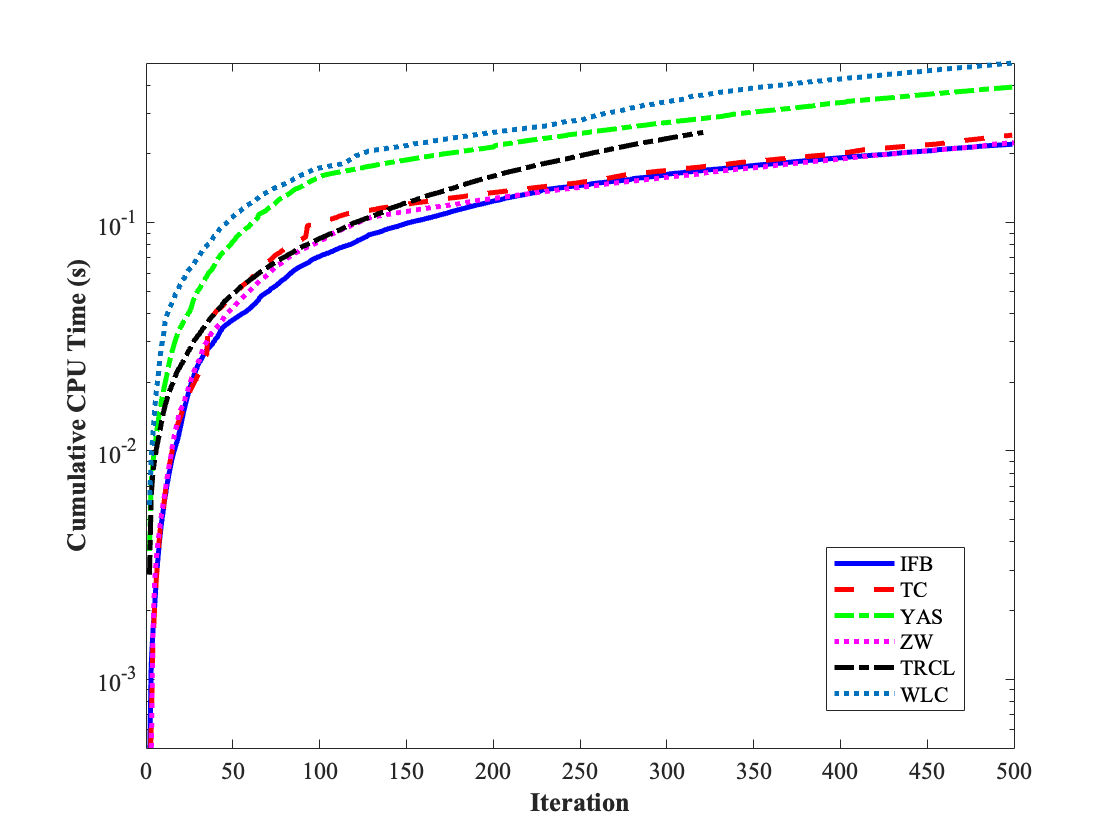}
			\caption{$d=512,\, m=256$}
		\end{subfigure}
		\begin{subfigure}[b]{0.49\linewidth}
			\includegraphics[width=\linewidth]{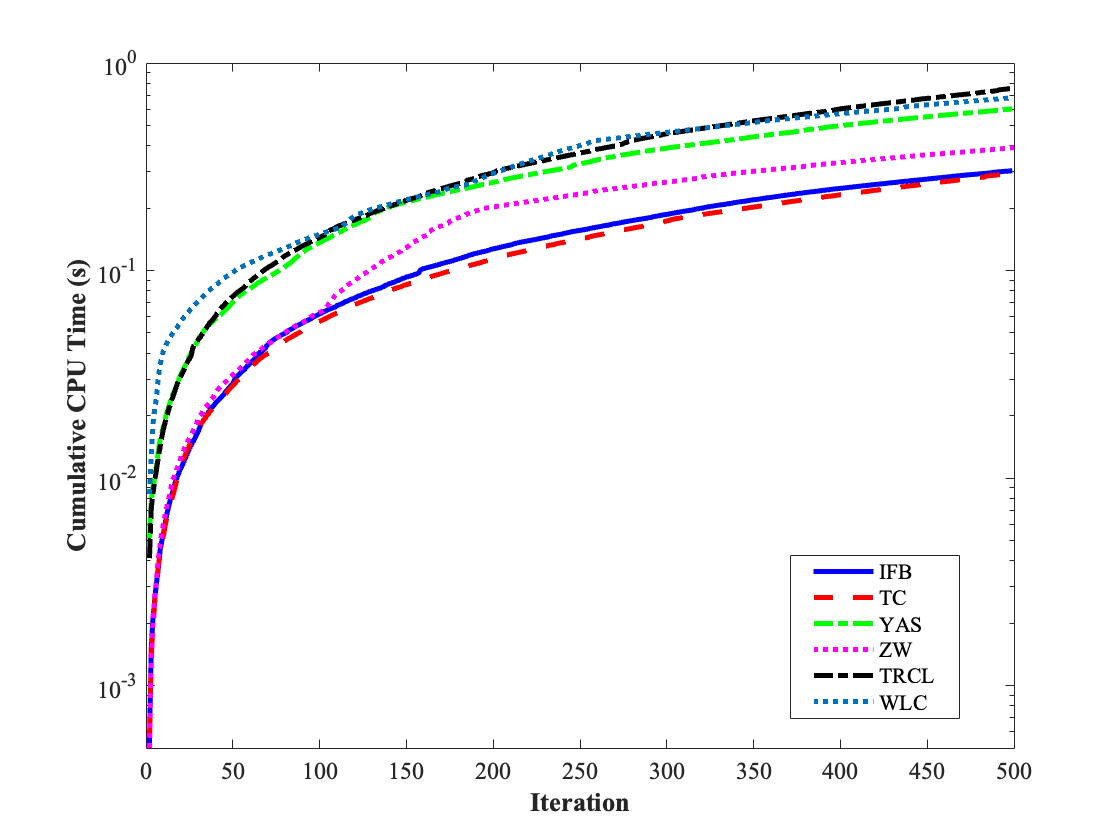}
			\caption{$d=1024,\, m=512$}
		\end{subfigure}
		\begin{subfigure}[b]{0.49\linewidth}
			\includegraphics[width=\linewidth]{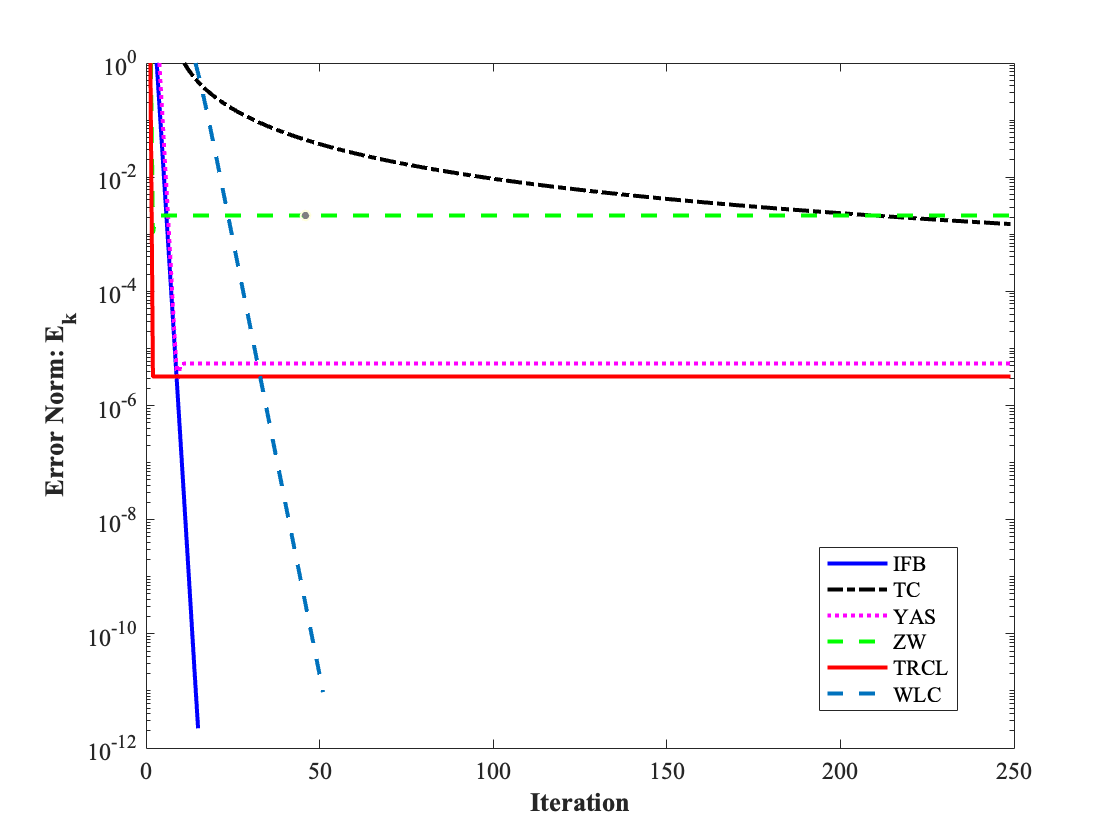}
			\caption{$d=512,\, m=256$}
		\end{subfigure}
		\begin{subfigure}[b]{0.49\linewidth}
			\includegraphics[width=\linewidth]{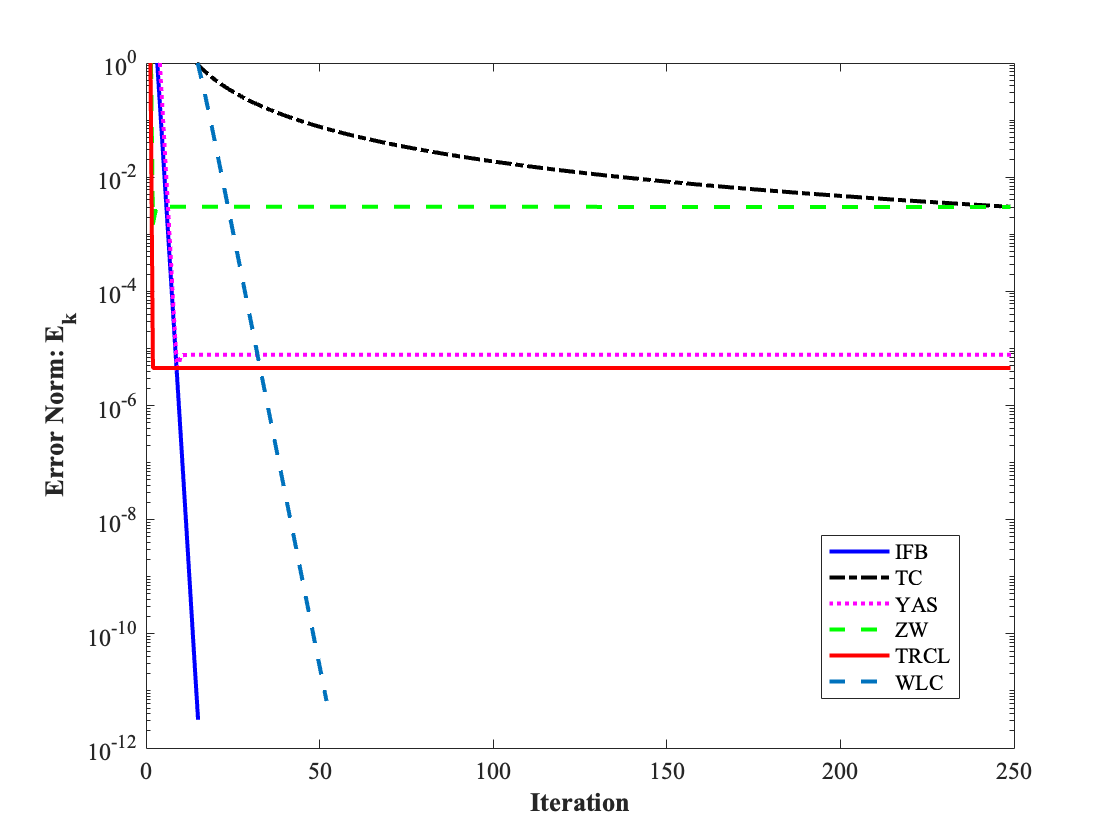}
			\caption{$d=1024,\, m=512$}
		\end{subfigure}
		\caption{Variation of CPU time and error $E_k$ of Example \ref{ex2}}
		\label{fig5}
	\end{figure}				
\end{example}

\begin{example}\label{ex3}
	Let \(\mathcal{H} := L^2[0,1] \) be a Hilbert space with the inner product and induced norm defined as
	\[
	\langle u, v \rangle := \int_0^1 u(t) v(t) \, dt  \quad \text{and} \quad \|u\| := \left( \int_0^1 u(t)^2 \, dt \right)^{1/2}, \quad \forall u, v \in\mathcal{H} \mbox{ and } \forall t \in [0, 1].
	\]
	See for more details \cite{BC11}. Define a convex, proper and lower semi-continuous function \( \phi : L^2[0,1] \to \mathbb{R} \cup \{+\infty\} \) by
	\[
	\phi(u) := \int_0^1 |u(t)| \, dt.
	\]
	Let \( \mathcal{A} := \partial \phi \) be denote the subdifferential of \( \phi \), then for every \( t \in [0,1] \), the set-valued operator \( \mathcal{A} : L^2[0,1] \rightrightarrows L^2[0,1] \) is defined as
	\[
	\mathcal{A}(u(t)) \in 
	\begin{cases}
		\{1\} & \text{if } u(t) > 0, \\
		[-1, 1] & \text{if } u(t) = 0, \\
		\{-1\} & \text{if } u(t) < 0.
	\end{cases}
	\]
	Since \( \mathcal{A} \) is the subdifferential of a convex function, it is the maximal monotone operator. 
	Define a single-valued operator \( \mathcal{B} : L^2[0,1] \to L^2[0,1] \) as
	\[
	\mathcal{B}(u(t)) := u(t) \cdot \log\left(1 + |u(t)|\right), \quad \forall  u \in L^2[0,1].
	\]
	The operator $\mathcal{B}$ is monotone, as the function \( f(u) = u \log(1 + |u|) \) is monotone increasing on \( \mathbb{R} \). However, $\mathcal{B}$ is not Lipschitz continuous due to the unbounded derivative of $f$ near \( u = 0 \).
	
	$E_k = \|u_{k+1} - u_k\| $ is used to calculate the iteration's accuracy. The following provides the stopping criterion
	\[
	E_k \leq 10^{-12}.
	\]
	We consider various initial values in the Table \ref{Table3}.
	\begin{table}[H]
		\centering
		\begin{minipage}{\textwidth}
			\rule{\textwidth}{1pt}
			\begin{tabular*}{\textwidth}{@{\extracolsep{\fill}}lcc@{}}
				Cases & $u_0$ & $u_1$ \\
				\hline
				&& \\
				1 & $\frac{\cos^2(2\pi t)}{4}$  & $\frac{3e^{-2t} \cos(3t)}{25}$   \\
				&& \\
				2           & $\frac{e^{2t} + \cos(4t)}{10}$ & $\frac{\cos^2(2\pi t)}{4}$  \\
				&&         \\
				3         & $\frac{e^{2t} + \cos(4t)}{10}$ & $\frac{3e^{-2t} \cos(3t)}{25}$   \\
				&&         \\
				4          & $\frac{\cos^2(2\pi t)}{4}$ & $\frac{e^{2t} + \cos(4t)}{10}$   \\
			\end{tabular*}
			\rule{\textwidth}{1pt}
		\end{minipage}
		\vspace{\abovecaptionskip}
		\caption{Initial values of Example \ref{ex3}}
		\label{Table3}
	\end{table}

	\begin{center}
		\begin{sideways}
			\begin{minipage}{\textheight}  
				\centering
				\begin{tabular}{@{\extracolsep{\fill}}ccccccccccccc@{}}
					\toprule
					\multirow{2}{*}{\textbf{Initial values}} & \multicolumn{3}{c}{\textbf{Case 1}} & \multicolumn{3}{c}{\textbf{Case 2}} & \multicolumn{3}{c}{\textbf{Case 3}} & \multicolumn{3}{c}{\textbf{Case 4}} \\
					\cline{2-3}\cline{4-6}\cline{5-7}\cline{8-10}\cline{11-13}
					& Iter. & CPU(s) & Error & Iter. & CPU(s) & Error & Iter. & CPU(s) & Error & Iter. & CPU(s) & Error \\
					\midrule
					IFB    & 15  & 0.03187 &   1.42e-12       & 16  &  0.04087 &1.81e-12  & 14 & 0.03939 &    1.58e-12   &    16   &   0.03679    &   1.68e-12    \\
					&&&&&&&&&&&&\\
					TC     & 250  & 0.17690 &   2.47e-05       & 251  &  0.14239 &1.32e-05  & 250 & 0.18234 &    2.34e-05   &    251   &   0.15979    &   2.34e-05    \\
					&&&&&&&&&&&&\\
					YAS    & 250  & 0.07390 &   3.12e-08       & 249  &  0.15690 &1.43e-08  & 252 & 0.08756 &    3.98e-08   &    254   &   0.05632    &   1.54e-08    \\
					&&&&&&&&&&&&\\
					ZW     & 250  & 0.14312 &   2.09e-05       & 248  &  0.27623 &2.16e-05  & 251 & 0.14434 &    1.41e-05   &    251  &   0.17853   &   1.22e-05    \\
					&&&&&&&&&&&&\\
					TRCL   & 250  & 0.12452 &   1.76e-08       & 250  &  0.15690 &3.28e-08  & 253 & 0.12905 &    2.90e-08   &    249   &   0.54324    &   3.64e-08    \\
					&&&&&&&&&&&&\\
					WLC   & 45  & 0.05409 &   1.23e-12       & 45  &  0.02309 &2.40e-12  & 44 & 0.01234 &    2.65e-12   &    46   &   0.044783    &   2.99e-12    \\
					&&&&&&&&&&&&\\
					\bottomrule
				\end{tabular}
				\captionof{table}{Performance comparison for various initial values of Example~\ref{ex3}}
				\label{Table4}
			\end{minipage}
		\end{sideways}
	\end{center}

	\begin{figure}[H]
		\centering
		\begin{subfigure}[b]{0.49\linewidth}
			\includegraphics[width=\linewidth]{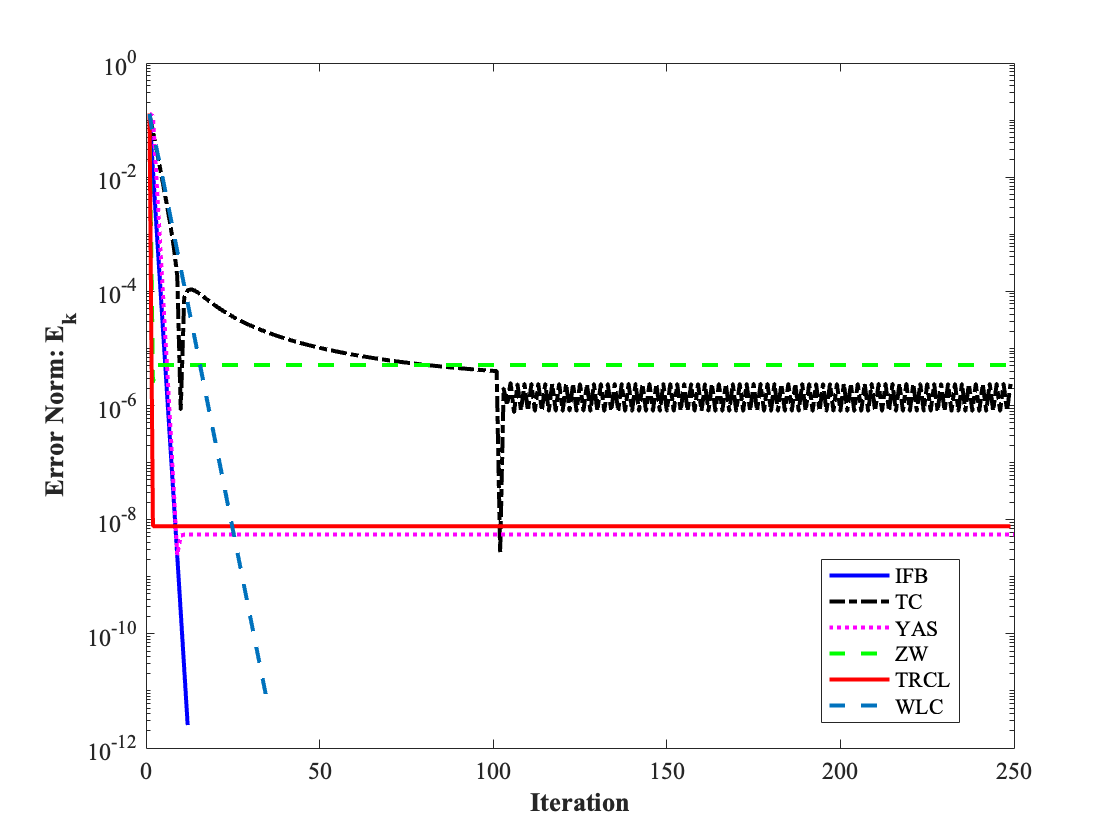}
			\caption{Cases 1}
		\end{subfigure}
		\begin{subfigure}[b]{0.49\linewidth}
			\includegraphics[width=\linewidth]{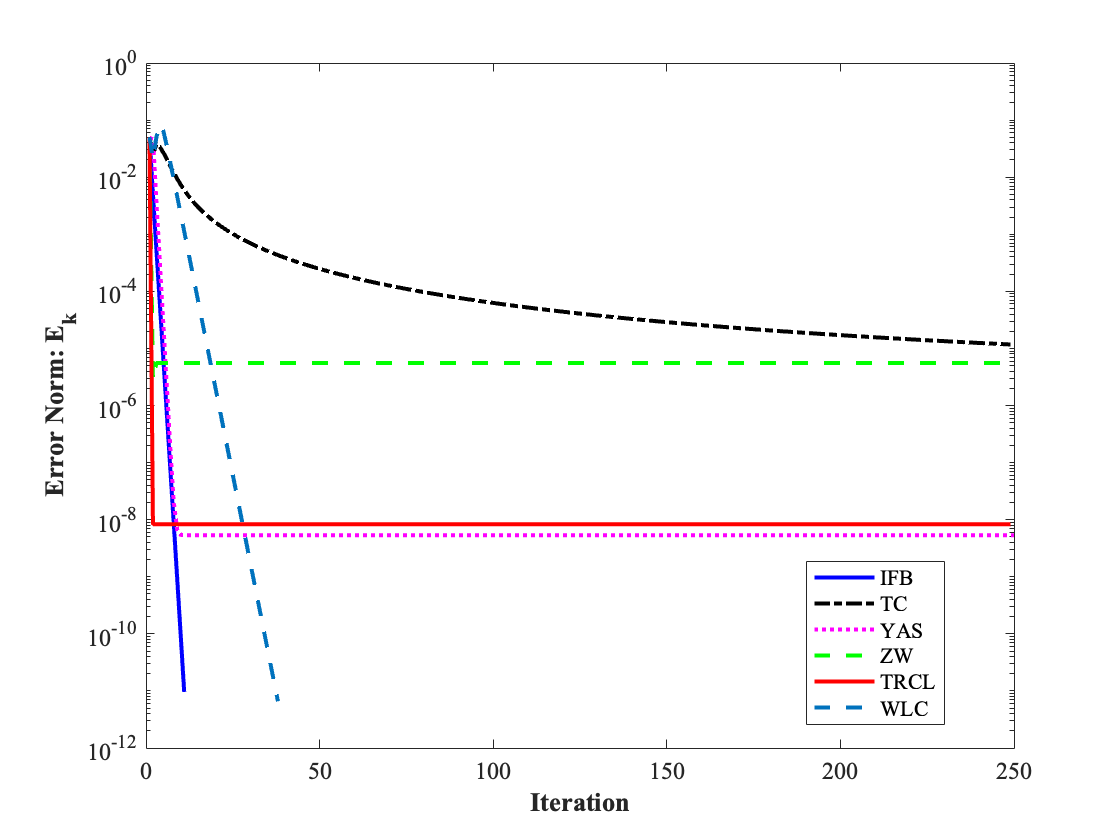}
			\caption{Cases 2}
		\end{subfigure}
		\begin{subfigure}[b]{0.49\linewidth}
			\includegraphics[width=\linewidth]{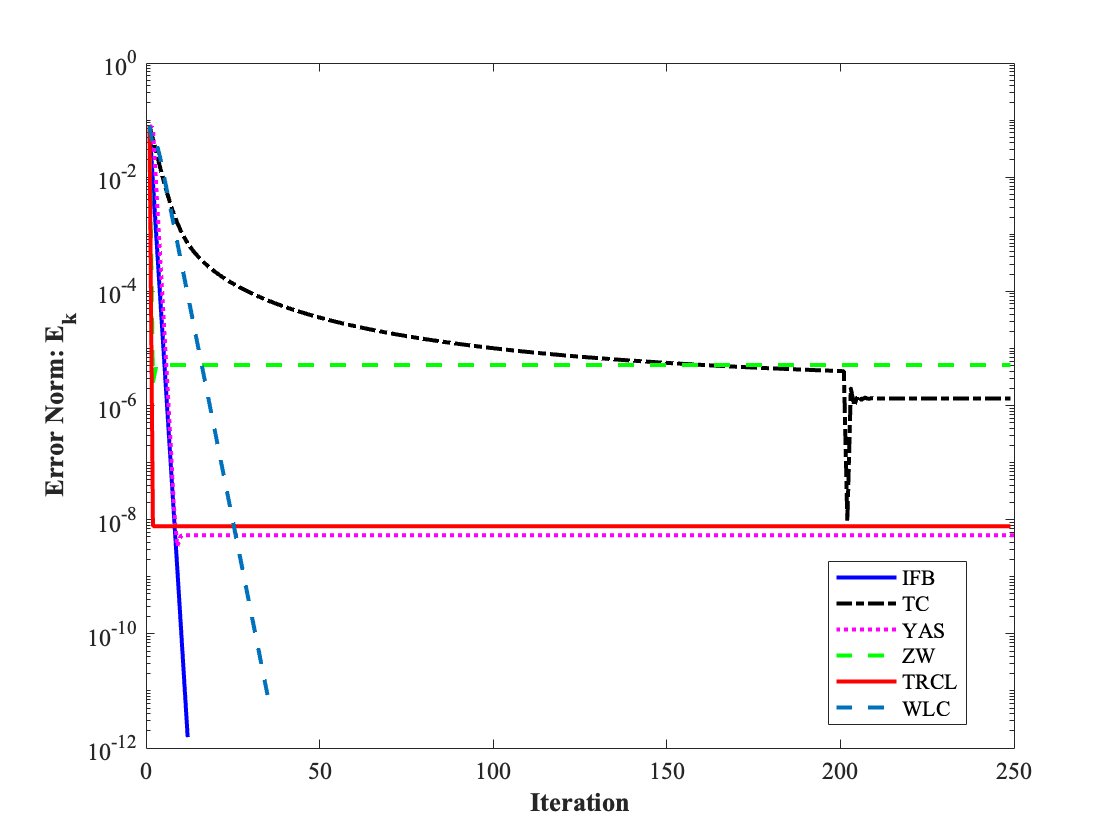}
			\caption{Cases 3}
		\end{subfigure}
		\begin{subfigure}[b]{0.49\linewidth}
			\includegraphics[width=\linewidth]{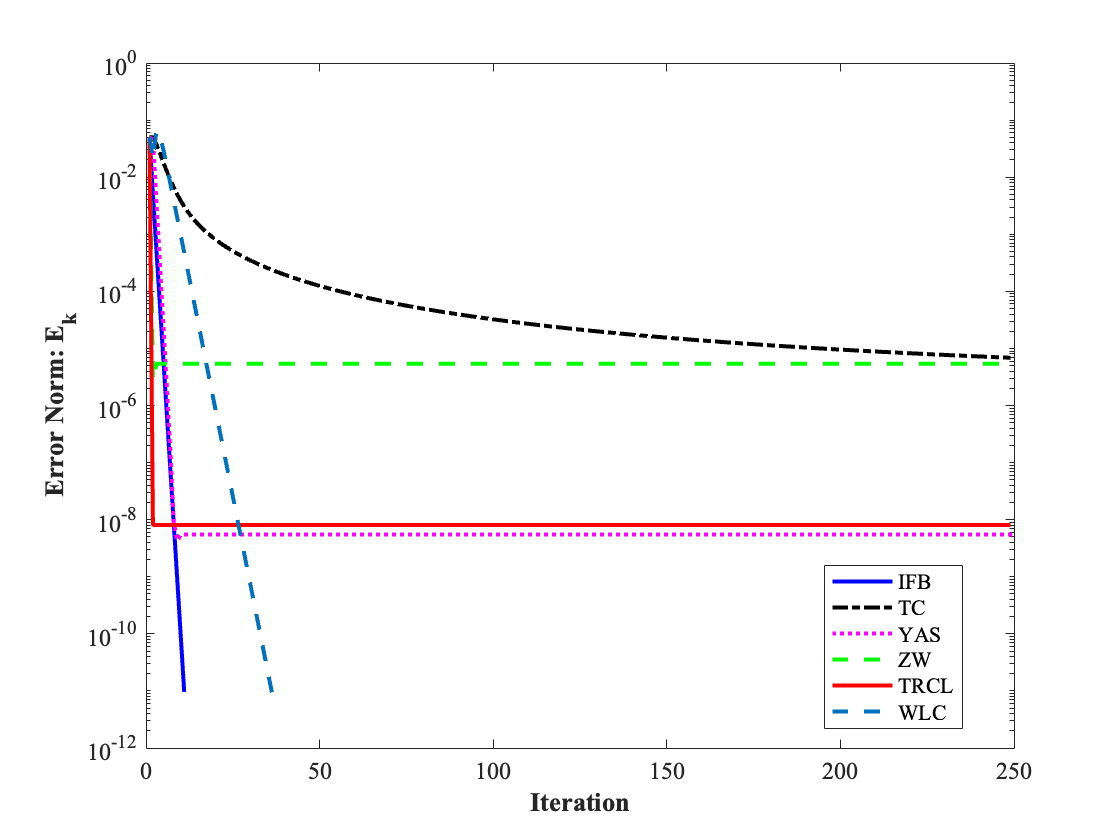}
			\caption{Cases 4}
		\end{subfigure}
		\caption{Convergence of $E_k$ of Example \ref{ex3}}
		\label{fig6}
	\end{figure}		
\end{example}

\section{Conclusion}
\noindent In this study, we developed an inertial contraction-type method for solving monotone variational inclusion problems (MVIP) in real Hilbert spaces without requiring the usual assumptions of coercivity or Lipschitz continuity on the single-valued operator. This relaxation of assumptions significantly broadens the applicability of the method. We established weak convergence with a rate of $\mathcal{O}(1/\sqrt{k})$, and under stronger assumptions, namely, maximal and strong monotonicity of the set-valued operator, we achieved strong convergence with a linear rate. Our numerical experiments on signal recovery problems validate the theoretical findings and highlight the algorithm's robustness. Notably, the performance improves as the relaxed parameter sequence $\{\vartheta_k\}$ increases, demonstrating the practical advantages of our approach in real-world scenarios where standard assumptions may not hold. These results emphasize the method's suitability for various optimization problems in modern applications. This study advances the current work by unifying and extending recent developments in iterative algorithms for optimization and inclusion problems. We aim to explore stochastic versions of our proposed method and further investigate potential acceleration strategies to enhance convergence rates.

\section{Conflict of interest}
All authors declare that they have no conflicts of interest.

\end{document}